\documentclass[12pt]{article}

\usepackage{graphicx}
\usepackage{latexsym,amssymb}
\usepackage{amsthm}
\usepackage{indentfirst}
\usepackage{amsmath}
\usepackage{color}
\usepackage{xcolor}
\usepackage{fourier}

\newtheorem{theoremalph}{Theorem}
\renewcommand\thetheoremalph{\Alph{theoremalph}}
\newtheorem*{Main Theorem}{Main Theorem}

\newtheorem{Theorem}{Theorem}[section]
\newtheorem*{Theorem A}{Theorem A}
\newtheorem*{Theorem A'}{Theorem A'}
\newtheorem*{Theorem B'}{Theorem B'}

\newtheorem{Definition}[Theorem]{Definition}
\newtheorem{Proposition}[Theorem]{Proposition}
\newtheorem{Lemma}[Theorem]{Lemma}

\newtheorem{Remark}[Theorem]{Remark}

\newtheorem{Corollary}[Theorem]{Corollary}

\newtheorem*{Claim}{Claim}
\newtheorem{Claim-numbered}[Theorem]{Claim}

 \def\NN{{\mathbb N}} 

 \def\RR{{\mathbb R}} 
\def\TT{{\mathbb T}}

 \def\ZZ{{\mathbb Z}}

\def\Si{\Sigma}
\def\La{\Lambda}

\def\cA{{\cal A}}  \def\cG{{\cal G}} \def\cM{{\cal M}} 
   \def\cN{{\cal N}} 
\def\cC{{\cal C}}  \def\cI{{\cal I}}  
   \def\cP{{\cal P}} \def\cV{{\cal V}}
    \def\cW{{\cal W}}
\def\cF{{\cal F}}  \def\cL{{\cal L}}

\newcommand{\Id}{\operatorname{Id}}

\newcommand{\diff}{{\operatorname{Diff}}}

\def\diff{\operatorname{Diff}}

\def\dim{\operatorname{dim}}

\def\diam{\operatorname{Diam}}

\def\supp{\operatorname{Supp}}
\def\ud{\operatorname{d}}
\def\e{{\varepsilon}}

\def\det{\operatorname{det}}

\def\Leb{\operatorname{Leb}}
\def\homeo{\operatorname{Homeo}}

\begin{document}

\title{Empirical measures of partially hyperbolic attractors}

\author{Sylvain Crovisier,  Dawei Yang and  Jinhua Zhang\footnote{S.C and J.Z were partially supported by  the ERC project 692925 \emph{NUHGD}.
D.Y  was partially supported by NSFC 11671288, 11822109,  11790274.}}


\maketitle


\begin{abstract}
In this paper, we study the limit measures of the empirical measures of Lebesgue almost every point in the basin of a partially hyperbolic attractor. They are strongly related to a notion named \emph{Gibbs $u$-state}, which can be defined in a large class of diffeomorphisms with less regularity and which is the same as Pesin-Sinai's notion for partially hyperbolic attractors of $C^{1+\alpha}$ diffeomorphisms.

In particular, we prove that for partially hyperbolic $C^{1+\alpha}$ diffeomorphisms with one-dimensional center,
and for Lebesgue almost every point: (1) the center Lyapunov exponent is well defined, but (2) the sequence of  empirical measures may not converge.

We also give some consequences on \emph{SRB measures} and \emph{large deviations}.

\hspace{-1cm}\mbox
\smallskip

\end{abstract}
\section{Introduction}
Let $f$ be a diffeomorphism of a closed manifold $M$. As a general goal of dynamical systems,
we are interested in describing the asymptotic behavior of the orbits of $f$.
In particular, it is expected (see~\cite{Ru4,T1,T2}) that, for most systems and Lebesgue almost every point $x\in M$,
one gets convergence as $n\to +\infty$ of
the sequence of \emph{empirical measures}
$$m_{x,n}:=\frac 1n \sum_{i=0}^{n-1} \delta_{f^i(x)},$$
although there exist examples of systems where for Lebesgue a.e. $x$ the limit does not exist
(Bowen has built such example inside the wandering set of a surface diffeomorphism, see~\cite{T1};
another example occurs inside the quadratic family on the interval, see~\cite{HK}).
In a second step, one may wonder if the set of limit measures (associated to points in a set with full Lebesgue measure)
is finite -- this is clearly not satisfied when $f$ is the identity map. This leads to the problem of the existence of a \emph{physical measure},
i.e.  an  $f$-invariant probability measure $\mu$
such that the set $\{x\in M:\; m_{x,n}\to \mu\}$ has positive Lebesgue measure.

In 1970s, Y. Sinai, D. Ruelle and R. Bowen ~\cite{Sinai,Bow75,Ru3}
have shown that uniformly hyperbolic $C^{1+\alpha}$ diffeomorphisms may be described by finitely many
physical measures satisfying additional geometrical properties and called \emph{SRB measures},
whereas these systems in general also possess many invariant probability measures that are not observable.
In this paper, we discuss systems satisfying a weaker form of hyperbolicity, called partial hyperbolicity.

\subsection{Empirical measures of partially hyperbolic attractors with 1D-center}
A diffeomorphism $f$ is $C^{1+\alpha}$, for $\alpha>0$, if it is $C^1$ and both $Df$ and $Df^{-1}$ are $\alpha$-H\"older.
Let $\Lambda$ be an \emph{attracting} compact set, i.e. it admits an open neighborhood $U$
such that $f(\overline{U})\subset U$ and $\La=\bigcap_{n\in\mathbb{N}} f^n(U)$.
Its attracting \emph{basin} is the open set $\bigcup_{n\in \mathbb{Z}}f^n(U)$.
The set $\Lambda$ is \emph{partially hyperbolic} if there exists an invariant dominated
splitting $T_\Lambda M=E^{ss}\oplus E^{c}\oplus E^{uu}$ of the tangent space over $\Lambda$,
such that $E^{uu}$ is uniformly expanded and $E^{ss}$ is uniformly contracted, see Section~\ref{Sec:definition-partial}.
(One of the extremal bundles may be degenerate and the splitting is then denoted by
$E^{cs}\oplus E^{uu}$ or $E^{ss}\oplus E^{cu}$).
When $\Lambda$ is attracting, the bundle $E^{cs}$ extends uniquely
as a continuous invariant bundle over a neighborhood of $\Lambda$.

Most of the works addressing existence of physical measures in the partially hyperbolic
setting assume  that the bundle $E^{cs}$ (or $E^{cu}$) satisfies some weak form of contraction
(or expansion), see for instance~\cite{ABV, BV}.
In this paper we first consider the case where the center $E^c$ is one-dimensional
and allow mixed behavior.
We prove that Lebesgue almost every point has a well defined center Lyapunov exponent.
We recall that a $f$-invariant probability measure $\mu$ is \emph{hyperbolic}
if for $\mu$-almost every $x\in M$ and any non-zero vector $v\in T_xM$,
the quantity $\frac 1 n \log \|Df^n(x).v\|$ does not converge to $0$ as $n\to \infty$.

\begin{theoremalph}~\label{thm-I}
Let $f$ be a $C^{1+\alpha}$ diffeomorphism of a closed manifold and $\La$ be an attracting set with a partially hyperbolic splitting
$T_\La M=E^{ss}\oplus E^c\oplus E^{uu}$ such that $\dim(E^c)=1$.
Then for Lebesgue almost every point $x$ in a neighborhood $U$ of $\Lambda$
the following limit exists:
	$$\lambda^c(x):=\lim_{n\rightarrow+\infty}\frac{1}{n}\log\|Df^n|_{E^{cs}(x)}\|.$$
Moreover, if $\lambda^c(x)\neq 0$, then $x$ is in the basin of a hyperbolic and ergodic physical measure.
\end{theoremalph}

When $\lambda^c(x)=0$, the sequence of empirical measures of $x$ may not converge,
as the following example shows. Contrary to Bowen's example mentioned above,
the dynamics is non-wandering. 

\begin{theoremalph}~\label{thm-II}
There exists a transitive diffeomorphism $f\in\diff^\infty(\mathbb{T}^3)$ with
a partially hyperbolic splitting $T\mathbb{T}^3=E^{ss}\oplus E^c\oplus E^{uu}$, $\dim(E^c)=1$,
such that Lebesgue almost every point $x\in\mathbb{T}^3$ has a dense orbit and   its sequence of empirical measures
	$\frac 1 n \sum_{i=0}^{n-1} \delta_{f^i(x)}$ does not converge.
\end{theoremalph}

\subsection{Gibbs u-states revisited}
We are aimed at studying the properties of the limits $\mu$ of the empirical measures $m_{x,n}$
(before discussing their uniqueness).
For instance, when $f$ is $C^{1+\alpha}$ for some $\alpha>0$  and preserves a volume $\mu$, Pesin ~\cite{P}  has shown that the entropy of $\mu$
is equal to the sum of its positive Lyapunov exponents.
This can be generalized as follows (see~\cite{CCE,CaY} and the appendix~\ref{Appendix-A}):

\paragraph{Generalized Pesin's inequality.}
\emph{For any $C^1$ diffeomorphism $f$, if $\Lambda$ is an invariant compact set with a dominated splitting $E\oplus F$,
then for Lebesgue almost every point $x$ satisfying $\omega(x)\subset \Lambda$,
the entropy of any limit measure $\mu$ of the sequence $\frac 1 n \sum_{i=0}^{n-1} \delta_{f^i(x)}$
is bounded from below:
\begin{equation}\label{e.pesin}
h_\mu(f)\geq \int \log|\det Df|_F|d\mu.
\end{equation}}
We stress that we only require $f$ to be $C^1$ and
the measure $\mu$ is not known a priori.
\begin{Remark}
Note that it has the following interesting consequence: for any $f\in \diff^1(M)$ and any fixed point $p$, if $|\det(Df(p)|>1$, then the Dirac measure $\delta_p$ is not physical.
\end{Remark}

When $\Lambda$ is an attracting set with a partially hyperbolic splitting
$T_\Lambda M=E^{cs}\oplus E^{uu}$, it contains each strong unstable leaf $\cF^{u}(x)$ of its points
and therefore is the support of a lamination denoted as $\cF^{u}$.
To any invariant measure $\mu$ supported on $\La$, 
an \emph{entropy $h_\mu(f,\cF^{u})$ along the strong unstable lamination $\cF^u$} is associated (see Definition~\ref{def.partial-entropy}):
in this setting this has been introduced by Yang in~\cite{Yan16} and for $C^2$-diffeomorphisms it coincides with Ledrappier-Young entropy~\cite{LY1} along the
invariant bundle $E^{uu}$.
Our next result shows that it satisfies an equality similar to Pesin's formula.

\newcounter{theorembisbis}
\setcounter{theorembisbis}{\value{theoremalph}}
\begin{theoremalph}\label{thm-III}
For any $C^1$ diffeomorphism $f$, if $\Lambda$ is an attracting set with a partially hyperbolic splitting $E^{cs}\oplus E^{uu}$, then there exists a small neighborhood $U$ of $\La$ such that  for Lebesgue almost every point $x\in U$, any limit $\mu$
of the sequence $\{\frac 1 n \sum_{i=0}^{n-1} \delta_{f^i(x)}\}$ satisfies
\begin{equation}\label{e.u-gibbs}
h_\mu(f,\cF^{u})=\int\log{|\det(D f|_{E^{uu}})|}\ud\mu,
\end{equation}
where $\cF^u$ is the strong unstable lamination on $\Lambda$ tangent to $E^{uu}$.
\end{theoremalph}

This immediately gives the following consequence.
\begin{Corollary}\label{coro-I}
Let $f$ be a $C^{1}$ diffeomorphism and $\La$ be an attracting set with a partially
hyperbolic splitting $ E^{cs}\oplus E^{uu}$. Assume that there exists a unique measure $\mu$ on $\La$ satisfying~\eqref{e.u-gibbs},then $\mu$ is a physical measure; moreover its basin has full Lebesgue measure in the basin of $\La$.
\end{Corollary}

\begin{Remark}
When there is more than one measure satisfying~\eqref{e.u-gibbs}, there may not be any physical measure,
as it is the case in Theorem~\ref{thm-II}.
\end{Remark}

This motivates the following definition.
\begin{Definition}
Let $\cF^u$ be an unstable lamination of a $C^1$-diffeomorphism $f$.
An invariant probability $\mu$ supported on  $\cF^u$ is a \emph{Gibbs u-state} if
it satisfies~\eqref{e.u-gibbs}.
\end{Definition}

When $f$ is $C^{1+\alpha}$, this property is known to be equivalent to
the fact that the disintegrations of $\mu$ along the unstable leaves
are absolutely continuous with respect to the Lebesgue measure along the leaves,
which is the original definition of Gibbs u-state given by Pesin and Sinai
(see for instance~\cite[Chapter 11]{BDV} and the Section~\ref{s.partialentropyalongexpandingfoliation}).
For $C^1$ diffeomorphisms however, an invariant measure may satisfy~\eqref{e.u-gibbs}
without having absolutely continuous disintegrations, see~\cite{RY,CQ,BMOS}.
For $C^1$ diffeomorphisms, the Gibbs u-states satisfy some properties (well-known
for smoother diffeomorphisms):
the set of Gibbs u-states is convex, compact for the weak-$*$ topology
and varies upper semi-continuously with respect to the systems in $C^{1}$-topology, see Section~\ref{s.partialentropyalongexpandingfoliation}.

Under the $C^{1+\alpha}$ smoothness hypothesis, and also assuming that $\Lambda=M$,
Theorem~\ref{thm-III} follows from~\cite[Theorem 11.15]{BDV} and Corollary~\ref{coro-I}
is~\cite[Corollary 2]{D}.
So our main improvements is to provide a different proof with no distortion arguments which applies to the $C^1$-case
and to show how it extends to the basin of $\Lambda$ (where the partially hyperbolic structure
does not exist in general). Proofs of generalized Pesin's formula have been obtained in various $C^1$ settings see~\cite{M,
CQ, Q, CCE, CaY} for instance; in our case we have to work with the entropy along an unstable lamination.
\medskip

If $\La$ is an invariant compact set admitting a partially hyperbolic splitting $E^{cs}\oplus E^{uu}$,
one says that a subset of $TM$ is an \emph{unstable cone field} $\mathcal{C}^{u}$
if there exists a continuous extension $T_UM=E\oplus F$ of $E^{cs}\oplus E^{uu}$
over a neighborhood $U$ of $\Lambda$, and a continous map $\theta\colon U\to (0,+\infty)$
such that for each $x\in U$ the set $\cC^{u}(x)=\cC^u\cap T_xM$ coincides with the cone:
$$\cC^{u}(x):=\{v=v^E+v^F\in E_x\oplus F_x:\;\; \theta\cdot \|v^F\|\geq \|v^E\|\}.$$

This allows to state a more general version of Theorem~\ref{thm-III}
for (not necessarily attracting) unstable laminations, which addresses the limit of empirical measures
for Lebesgue almost every point $x$ in any disc tangent to an unstable cone fields
(rather than almost every point whose forward orbit stays in a neighborhood of $\Lambda$),
see Theorem~\ref{thm-III-bis} in Section~\ref{s.Proofs}.
As a consequence, we prove that the construction of Gibbs u-states for $C^{1+\alpha}$-diffeomorphisms
done by Pesin and Sinai~\cite{PS} can be adapted to $C^1$-diffeomorphisms.

\begin{Corollary}\label{coro-II}
Consider a $C^1$ diffeomorphism $f$, an attracting set $\Lambda$ with a partially hyperbolic splitting $ E^{cs}\oplus E^{uu}$
and an unstable cone field $\cC^u$.
Then there exists a neighborhood $U$ of $\Lambda$ such that
for any probability measure $\Leb^*_{D}$ which is the normalized Lebesgue measure
on a disc $D\subset U$ tangent to $\cC^{u}$,
each limit measure $\mu$ of the sequence
$$\frac{1}{n}\sum_{i=0}^{n-1}f^i_*\Leb^*_{D}$$
satisfies the entropy formula~\eqref{e.u-gibbs}.
\end{Corollary}

Another important class of measures related to their observability are SRB measures.
\begin{Definition}
An invariant probability $\mu$ of a $C^1$ diffeomorphism $f$ is an \emph{SRB measure} if
$$h_{\mu}(f)=\int \sum\lambda^+(z)\ud\mu(z),$$
where $\sum\lambda^+(z)$ is the sum of all the positive Lyapunov exponents of $z$
(with multiplicities).
\end{Definition}
For $C^{1+\alpha}$ diffeomorphisms, this is equivalent to require that
the disintegrations of $\mu$ along its unstable manifolds are absolutely continuous
(see~\cite{LY1,Brown} but we will not use this fact).

\begin{Corollary}\label{coro-III}
Consider a $C^1$ diffeomorphism $f$ and an attracting set $\La$ with a partially hyperbolic splitting
$T_\La M=E^{ss}\oplus E^c\oplus E^{uu}$ such that $\dim(E^c)=1$.
Then for Lebesgue almost every point $x$ in a neighborhood of $\Lambda$,
the $\omega$-limit of $x$ supports an SRB measure.
\end{Corollary}

This extends~\cite{CoYo} which proves (using random perturbations)
that for $C^2$ diffeomorphisms, attracting sets that are partially hyperbolic with one-dimensional center support
an SRB measure.

\subsection{Large deviations}
Our approach  can also be used for bounding the large deviations for $C^1$-partially hyperbolic attracting sets with respect to continuous functions. (Theorem~\ref{thm-III} can also be deduced from that result by applying it to a countable and dense subset of $C^0(M,\mathbb{R})$.)

\newcounter{theorembisbisD}
\setcounter{theorembisbisD}{\value{theoremalph}}
\begin{theoremalph}\label{thm-IV}
Let $f$ be a $C^1$-diffeomorphism and
$\La$ be an attracting set with a partially hyperbolic splitting $T_\La M=E^{cs}\oplus E^{uu}$. 
Then there exists a small neighborhood $U$ of $\La$ such that for any continuous function $\varphi: M\to\mathbb{R}$ and  any $\e>0$, there exist $a_\e,b_\e>0$ such that 
	$$\Leb\bigg\{ x\in U: \ud\bigg(\frac{1}{n}\sum_{i=0}^{n-1}\varphi(f^i(x)), I(\varphi)\bigg)\geq\e\bigg\}<a_\e\cdot e^{-n b_\e} \quad\textrm{ for any  $n\in\mathbb{N}$},$$
	$$\text{where}\quad I(\varphi):=\bigg\{\int\varphi\ud\mu:\;\; \textrm{$\mu\in\mathcal{M}_{\rm inv}(\La,f)$ satisfies $h_\mu(f,\cF^u)=\int\log|\det(Df|_{E^{uu}})|\ud\mu$} \bigg\}.$$
\end{theoremalph}

Some results on the existence of SRB measures and the large deviation property for  singular hyperbolic attractors are obtained in Appendix~\ref{Sec:flow}.

\subsection*{Organization of the paper}
This paper proceeds as follows.
 In Section~\ref{s.preliminary}, we state the known results and notions used in the paper. 
 In Section~\ref{s.measurable-partitions-to-unstable}, we build increasing  measurable partitions subordinate to the strong unstable foliations and finite partitions approaching the measurable partition.
 In Section~\ref{s.volume-estimate}, we state and prove an intermediate result to Theorem~\ref{thm-III}.
 In Section~\ref{s.Proofs}, we firstly give the proof of a stronger version of Theorem~\ref{thm-III}
and we use it to give the proofs of Corollaries~\ref{coro-II} and~\ref{coro-III}. Then we prove our large deviations results. In Section~\ref{s.example},
we conclude the proof of (a stronger version of) Theorem~\ref{thm-I} and we build the example (Theorem~\ref{thm-II}).
Appendix A is devoted to extending the entropy inequality obtained in~\cite[Theorem 1]{CCE} to a semi-local setting,
whereas Appendix~\ref{Sec:flow} uses the results in Appendix A to prove the existence of physical measures for singular hyperbolic attractors
of $C^{1+\alpha}$-vector fields and a large deviations result.
\medskip

\noindent{\it Acknowledgments.}
We would like to thank Christian Bonatti, Yongluo Cao, Dmitry Dolgopyat, Fran\c{c}ois Ledrappier, Davi Obata, Yi Shi and Yuntao Zang for useful comments and discussions. {In a conference in Shenzhen in 2018, Y. Hua and J. Yang kindly mentioned us personally they are also working in a similar direction of a result in this paper. In particular~\cite[Theorems A and C]{HYY} are similar to the Theorems A and C here.}

\section{Preliminary}~\label{s.preliminary}

In this section, we collect the basic notions, tools and  known results used in this paper.
\subsection{Partial hyperbolicity}\label{Sec:definition-partial}

Let $f$ be a $C^1$-diffeomorphism of a closed manifold $M$.
An invariant splitting $T_\Lambda M=E^{cs}\oplus E^{uu}$ of the tangent bundle over an invariant compact set
$\La$ is  \emph{partially hyperbolic},
if there exists $N\in\mathbb{N}$ such that  
$$\|Df^N|_{E^{cs}(x)}\|\cdot \|Df^{-N}|_{E^{uu}(f^N(x))}\|\leq \frac{1}{2} \textrm{\:\: and \:\:} \| Df^{-N}|_{E^{uu}(x)}\| \leq \frac 12. $$
The bundle $E^{cs}$ then extends uniquely as a continuous invariant bundle
on the set of points whose forward orbit is included in a neighborhood of $\Lambda$
(as the limit of the backward iterates of a center-stable cone field, see~\cite[Chapter 2]{CP}),
moreover, each point $x\in \Lambda$ belongs to an injectively immersed submanifold $\cF^u(x)$ tangent to $E^{uu}(x)$,
and called strong unstable manifold.
One sometimes also assumes a finer invariant decomposition of the center-stable bundle
$E^{cs}:=E^{ss}\oplus E^c$ which satisfies:
$$\|Df^N|_{E^{ss}(x)}\|\cdot \|Df^{-N}|_{E^{c}(f^N(x))}\|\leq \frac{1}{2} \textrm{\:\: and \:\:} \| Df^{N}|_{E^{ss}(x)}\| \leq \frac 12. $$

A \emph{u-laminated set} is a $f$-invariant compact set $\Lambda$
endowed with a partially hyperbolic splitting $TM|_{\Lambda}=E^{cs}\oplus E^{uu}$
which satisfies the following property:
the (strong) unstable manifold
$\cF^u(x)$ at each point $x\in\Lambda$ tangent to $E^{uu}(x)$ is contained in $\Lambda$
(this is the case if $\Lambda$ is an attracting set).
The collection of unstable manifolds
defines a lamination called \emph{unstable lamination} associated to the u-laminated set $\Lambda$;
it is denoted by $\cF^u$. For each $x\in \Lambda$ and $\rho>0$,
we denote by $\cF^u_\rho(x)$ the ball in $\cF^u(x)$ centered at $x$ and of radius $\rho$.

\begin{Remark}\label{r.cone}
If $\mathcal{C}^u_1$, $\mathcal{C}^u_2$ are unstable cone fields on a neighborhood of $\Lambda$,
the domination implies that there exist a neighborhood $U$ of $\Lambda$ and $N\geq 1$
such that for any $x\in U\cap f^{-1}(U)\cap\dots\cap f^{-N}(U)$, we have
$Df^N(x)\mathcal{C}^u_1(x)\subset \mathcal{C}^u_2(f^N(x))$.
\end{Remark}

\subsection{Probability measures}

Let $X$ be a compact metric space. We recall that the space of probability Borel measures supported on $X$ is a compact metric space: consider a countable dense subset $\{\varphi_n\}_{n=0}^\infty$ in $C^0(X,\mathbb{R})$;
then the distance between two probability measures $\mu,\nu$ is given by
$$\ud(\mu,\nu):=\sum_{n=0}^\infty\frac{|\int\varphi_n\ud\mu-\int\varphi_n\ud\nu|}{2^n\cdot\sup_{x\in X}|\varphi_n(x)|},$$
and this gives the weak$*$-topology on the space of probability measures.

 \subsection{Pseudo-physical measures}
Let $f$ be a homeomorphism on a compact manifold $M$ and   $\cM_{\rm inv}(f)$ (or $\cM_{\rm inv}(M,f)$) be the set of $f$-invariant probability measures.
As before, given a point $x\in M$ we denote by $\mathcal{M}(x)\subset \cM_{\rm inv}(f)$ the set of accumulation points of 
the measures $\frac{1}{n}\sum_{i=0}^{n-1}\delta_{f^i(x)}$ as $n\to +\infty$.

For any $\mu\in \cM_{\rm inv}(f)$, we define its \emph{basin} to be 
$${\rm Basin}(\mu):=\{x\in M:\cM(x)=\{\mu\}\}.$$
The measure $\mu$ is said to be \emph{physical} if $\Leb({\rm Basin}(\mu))>0$.

We will use
a more general notion, introduced in~\cite{CE,CCE}.
The invariant measure $\mu$ is \emph{pseudo-physical} if for any $\eta>0$, one has 
  $$\Leb(\{x\in M: \ud(\cM(x),\mu)<\eta \})>0,$$
i.e. there exists a limit measure $\nu\in \cM(x)$ which is $\eta$-close to $\mu$.

A pseudo-physical measure is not necessary a physical measure.
In general, for a system, physical measures might not exist, however there always exist pseudo-physical measures. 
\begin{Theorem}[Theorems 1.3 and 1.5 in \cite{CE}]~\label{thm.pseudo-physical}
	Let $f\in\homeo(M)$. The set of pseudo-physical measures is non-empty and compact. Moreover, for Lebesgue a.e. $x\in M$, the set $\cM(x)$ is contained in the set of pseudo-physical measures.
\end{Theorem}

Let $D$ be an embedded compact $C^1$-disc in $M$. Then $\mu$ is called a \emph{pseudo-physical measure relative to $D$},
if for any $\eta>0$, one has 
$$\Leb_D(\{x\in D: \ud(\cM(x),\mu)<\eta \})>0.$$
Theorem~\ref{thm.pseudo-physical} is generalized as follows:
\begin{Theorem}~\label{thm.relative-pseudo-physical}
Let $f\in\homeo(M)$ and $D$ be an embedded compact $C^1$-disc. Then the set of pseudo-physical measures relative to $D$ is a compact non-empty set. Moreover, for Lebesgue a.e. $x\in D$, the set $\cM(x)$ is contained in the set of  pseudo-physical measures relative to $D$.
	\end{Theorem}
\proof
By definition, $\mu$ is not pseudo-physical if and only if there exists $\eta_\mu>0$ such that
$$
\Leb_D(\{x\in D: \ud(\cM(x),\mu)<\eta_\mu\})=0.
$$
Then any measure $\nu$ such that $\ud(\nu,\mu)<\eta_\mu/2$ is not pseudo-physical either
(take $\eta_\nu=\eta_\mu/2$). This proves the compactness.

We now denote by $\cP_D$ the set of   pseudo-physical measures relative to $D$
and consider its complement $\cP_D^c$  in $\cM_{\rm inv}(f)$.
Then $\cP_D^c=\cup_{n=1}^\infty A_n$, where $A_n:=\{\mu: \ud(\mu,\cP_D)\geq \frac{1}{n} \}$.
We define
$$W_n:=\{x\in D: \cM(x)\cap A_n\neq\emptyset \}.$$

Since each measure in $A_n$ is not pseudo-physical relative to $D$  and $A_n$ is compact, there exist $\mu_1,\cdots, \mu_l$ together with $l$ positive  numbers $\eta_1,\cdots,\eta_l$ such that 
\begin{itemize}
	\item $A_n\subset\bigcup_{i=1}^lB_{\eta_i}(\mu_i)$;
	\item $\Leb_D(\{x\in D: \ud(\cM(x),\mu_{i})<\eta_{i}\})=0$ for each $i$.
	\end{itemize}
This implies that $\Leb_D(W_n)=0$ for each $n$ and then $\cM(x)\subset \cP_D$ for Lebesgue a.e. $x\in D$.
\endproof

\begin{Remark}\label{r.localize}
This statement can be localized: if $X\subset D$ is a measurable subset that
has positive measure for $\text{Leb}_D$,
then $\mu$ is called a \emph{pseudo-physical measure relative to $X$},
if for any $\eta>0$,
$$\Leb_D\big(\big\{x\in X: \ud(\cM(x),\mu)<\eta \big\}\big)>0.$$
The proof of Theorem~\ref{thm.relative-pseudo-physical} also shows that
for Lebesgue a.e. $x\in X$, the set $\cM(x)$ is contained in the set of  pseudo-physical measures relative to $X$.
\end{Remark}

\subsection{Entropy for a general measurable partition}
In this part, we recall the notions of a measurable partition, entropy of a measurable partition, and their properties  from~\cite[$\S1-\S5$ and $\S9$]{Ro}.

If $\alpha$ is a partition of $X$, we denote $\alpha(x)$ the element of $\alpha$ which contains $x\in X$.
We denote $\alpha\prec \beta$ if $\beta(x)\subset \alpha(x)$ for each $x\in X$.
And if $(\alpha_i)_{i\in I}$ is a family of partitions, we denote by $\vee_{i\in I} \alpha_i$
the partition it generates, i.e. the partition $\alpha$ whose elements $\alpha(x)$
coincides with $\cap_i \alpha_i(x)$.
When $X$ is a metric space, the diameter of $\alpha$ is $\diam(\alpha)=\sup_{x\in X}\diam(\alpha(x)).$
\medskip

Let $\alpha$ be a partition of a Borel space $(X,\mathfrak{B})$. It is a \emph{measurable partition}, if there exists a sequence of finite measurable partitions $\alpha_1\prec\alpha_2\prec\cdots\prec\alpha_n\prec\cdots$ such that $\alpha=\bigvee_{i\in\mathbb{N}}\: \alpha_i$.
\medskip

Let $(X,\mathfrak{B},\mu)$ be a Lebesgue space and $\alpha$ be a measurable partition. We denote by $\mathfrak{B}_\alpha$ the $\sigma$-algebra of the Lebesgue space $X/\alpha$.  Then for $\mu$-a.e. $x\in X$, there exists a probability measure $\mu_x^\alpha$ supported on $\alpha(x)$ such that for any measurable set $A$ of $X$:
\begin{itemize}
\item  the map $x\mapsto \mu_x^\alpha(A)$ is $\mathfrak{B}_\alpha$-measurable;
\item  $\mu(A)=\int \mu_x^\alpha(A)\ud\mu(x).$
\end{itemize}
The probability measures $\mu_x^\alpha$ are called \emph{conditional measures} of $\mu$ with respect to $\alpha$.

Let $A_1,A_2,\cdots,$ be all the elements of $\alpha$ with positive $\mu$-measure.  The entropy $H_\mu(\alpha)$ of the measurable partition $\alpha$ is defined by

\begin{displaymath}
H_\mu(\alpha)= \left\{ \begin{array}{ll}
-\sum_{i=1}^\infty\mu(A_i)\cdot\log{\mu(A_i)} & \textrm{if $\mu(\cup_{i=1}^\infty A_i)=1$}\\
\infty & \textrm{otherwise.}
\end{array} \right.
\end{displaymath}

Let us consider another measurable partition $\beta$. Then $\alpha$ induces a partition $\alpha|_B$ of each element $B\in\beta$.
If $\mu_x^\beta$ denotes the conditional measures with respect to $\beta$, then the \emph{mean conditional entropy of  $\alpha$ with respect to $\beta$} is defined as
$$H_\mu(\alpha|\beta)=\int H_{\mu_x^\beta}(\alpha|_{\beta(x)})\ud\mu(x).$$
For measurable partitions, one has the following result:

\begin{Lemma}[5.9 in~\cite{Ro}]\label{Lem:H-condition}
For any three measurable partitions $\alpha,\beta,\gamma$, we have
$$H_\mu(\alpha\vee\beta|\gamma)=H_\mu(\alpha|\gamma)+H_\mu(\beta|\alpha\vee\gamma).$$
\end{Lemma}

\begin{Lemma}[5.7 and 5.11 in~\cite{Ro}]\label{l.updown}
Let $\alpha_1\prec\alpha_2\prec\cdots\prec\alpha_n\prec \cdots$ be an increasing sequence of measurable  partitions and $\beta$ be another measurable partition, then
\begin{enumerate}
\item
$H_{\mu}(\alpha_n|\beta)\nearrow H_\mu\big(\bigvee_{i=1}^\infty\alpha_i|\beta\big);$
\item if $H_{\mu}(\beta|\alpha_1)<\infty$, then
$H_{\mu}(\beta|\alpha_n)\searrow H_\mu\big(\beta|\bigvee_{i=1}^\infty\alpha_i \big).$
\end{enumerate}
\end{Lemma}
\medskip

Let $f$ be a homeomorphism of a compact metric space $X$ preserving a probability measure $\mu$.
Then $f$ is an automorphism of the Lebesgue space $(X,\mathfrak{B},\mu)$,
where $\mathfrak{B}$ denotes its Borel $\sigma$-algebra.
One defines the entropy $h_\mu(f,\alpha)$ with respect to a measurable partition $\alpha$:
$$h_\mu(f,\alpha):=H_\mu\left(\bigvee_{i=0}^\infty f^i(\alpha)\big|\bigvee_{i=1}^\infty f^i(\alpha)\right)=H_\mu\left(\alpha\big|\bigvee_{i=1}^\infty f^i(\alpha)\right).$$
A standard argument based on Lemma~\ref{Lem:H-condition} (see for instance~\cite[\S 7.3]{Ro}.) gives the following:
\begin{Lemma}~\label{l.alternative}
If $\alpha$ is a measurable partition such that $H_\mu(\alpha|f(\alpha))<\infty$, then
$$h_\mu(f,\alpha)=\inf \frac{1}{m}H_\mu\left(\bigvee_{j=1}^m f^{-j}(\alpha)\big|\alpha\right)=\liminf_{m\rightarrow\infty}\frac{1}{m}H_\mu\left(\bigvee_{j=1}^m f^{-j}(\alpha)\big|\alpha\right).$$
\end{Lemma}

One can now define the metric entropy (see~\cite[9.1]{Ro}):
\begin{eqnarray*}
	h_\mu(f)&=&\sup\{h_\mu(f,\alpha): \textrm{$\alpha$ is a finite measurable partition of $X$}  \}
	\\
	&=&\sup\{h_\mu(f,\alpha): \textrm{$\alpha$ is a measurable partition of $X$}  \}.
\end{eqnarray*}

The following property is obtained by applying inductively Lemma~\ref{Lem:H-condition}.
\begin{Corollary}\label{c.H-condition}
For any probability measure $\nu$, any sequence of finite measurable  partitions $\{\alpha_i\}_{i\in \mathbb{N}}$ and any integers $0\leq \ell< m<n$,
\begin{eqnarray*}
	H_\nu\big(\bigvee_{i=0}^{n-1}f^{-i}(\alpha_i)\big)
&=&H_\nu\big(\bigvee_{i=0}^{\ell}f^{-i}(\alpha_i)\big)\\
&&\quad\quad +\sum_{k=0}^{[\frac{n-\ell}{m}]-1}H_{f^{\ell+km}_*\nu}\big(\bigvee_{i=1}^{m}f^{-i}(\alpha_{i+\ell+km})|\bigvee_{i=0}^{km+\ell}f^{km+\ell-i}(\alpha_i)\big)
\\
&&\quad\quad\quad\quad\quad\quad\quad\quad+H_{\nu}\big(\bigvee_{i=\ell+[\frac{n-\ell}{m}]m+1}^{n-1}f^{-i}(\alpha_i)|\bigvee_{i=0}^{\ell+[\frac{n-\ell}{m}]m}f^{-i}(\alpha_i)\big)
.
\end{eqnarray*}
	\end{Corollary}

\subsection{Entropy along an unstable lamination}\label{s.partialentropyalongexpandingfoliation}
In this paper, we focus on the entropy of an invariant measure along an unstable lamination
as introduced in~\cite{VY} and~\cite{Yan16}.
Throughout this section, $f$ is a $C^1$-diffeomorphism of a compact manifold $M$ and
$\La$ is a u-laminated set. As before the associated unstable lamination is denoted by $\cF^u$.
We consider a probability measure $\mu$ supported on $\Lambda$.

A partition $\alpha$ of $M$ is \emph{$\mu$-subordinate to the unstable lamination $\cF^{u}$ of $\Lambda$},
if for $\mu$-a.e.~$x$,
\begin{itemize}
\item $\alpha(x)$ is contained in the strong unstable leaf $\cF^u(x)$ of the point $x$, and
\item $\alpha(x)$ contains an open neighborhood of $x$ in $\cF^u(x)$.
\end{itemize}
A partition $\alpha$ is \emph{increasing} if $f(\alpha(x))\supset \alpha(f(x))$ for  all $ x\in M$.

The existence of an increasing measurable partition $\mu$-subordinate to the unstable lamination is guaranteed by \cite[Proposition 3.1]{LS}  and ~\cite[Lemma 3.2]{Yan16}:
\begin{Lemma}\label{l.increasing}
For any $\mu\in\mathcal{M}_{\rm inv}(\La,f)$ there exists an increasing measurable partition which is $\mu$-subordinate to the unstable lamination $\cF^u$.
\end{Lemma}

The following result is an adapted version of Lemma 3.1.2 in ~\cite{LY}.

\begin{Lemma}
For any $\mu\in\mathcal{M}_{\rm inv}(\La,f)$ and any two increasing measurable partitions $\alpha_1,\alpha_2$ that are $\mu$-subordinate to the unstable lamination $\cF^{u}$,  one has
$h_{\mu}(f,\alpha_1)=h_{\mu}(f,\alpha_2).$
\end{Lemma}

One can thus define the entropy along the unstable lamination as follows.
\begin{Definition}\label{def.partial-entropy} The \emph{entropy of $\mu$ along the unstable lamination $\cF^u$} is
$$h_{\mu}(f,\cF^u)=h_{\mu}(f,\alpha),$$
where $\alpha$ is any increasing measurable partition $\mu$-subordinate to $\cF^u$.
\end{Definition}

\begin{Remark}\label{r.partial entropy} \begin{enumerate}
\item By definition $h_\mu(f,\cF^u)\leq h_\mu(f)$.
\item By~\cite[Proposition 2.14]{HHW}, $\mu\mapsto h_\mu(f,\cF^u)$ is affine from $\cM_{\rm inv}(\La,f)$ to $[0,\infty)$.
\item The notion of entropy along an unstable lamination is literally different from the one defined in~\cite[Section 7.2]{LY1}.
It has been proved that these two notions are the same in the $C^{1+\alpha}$-partially hyperbolic setting, $\alpha>1$. See ~\cite[Proposition 2.4]{VY} for a precise statement.
\end{enumerate}
\end{Remark}

The entropy along an unstable lamination satisfies an inequality generalizing Ruelle's one~\cite{Ru2}.

\begin{Theorem}[Theorem A in~\cite{WWZ}]\label{thm.ruelle}
Let $f$ be a $C^1$ diffeomorphism and $\La$ be a u-laminated set.
Then for any invariant measure $\mu\in\cM_{\rm inv}(\La,f)$, one has
$$h_\mu(f,\cF^u)\leq\int\log{|\det(Df|_{E^{uu}})|}\ud\mu.$$
\end{Theorem}

The entropy along an unstable lamination varies upper semi-continuously. This result is due to ~\cite{Yan16} (see also~\cite[Proposition 2.15]{HHW}.) 
\begin{Theorem}\label{Thm:usc}
Let $f\in\diff^1(M)$ be a partially hyperbolic diffeomorphism and $\{\mu_n\}$ be a sequence of $f$-invariant measures. Assume that $\mu_n$ converges to $\mu$ in weak$*$-topology,
then
$$\limsup_{n\rightarrow\infty} h_{\mu_n}(f,\cF^u)\leq h_\mu(f,\cF^u).$$
\end{Theorem}

One gets the following consequence from the previous results.
\begin{Corollary}\label{c.convex-compact}
Let $f$ be a $C^1$ diffeomorphism and $\La$ be a u-laminated set.
Then the set
\begin{equation}
\mathcal{M}_{\rm u}:=\big\{\mu\in \cM_{\rm inv}(\Lambda,f):\; h_\mu(f,\cF^u)=\int\log{|\det(Df|_{E^{uu}})|}\ud\mu\big\}
\end{equation}
is convex and compact. A measure belongs to $\mathcal{M}_{\rm u}$ iff each of its ergodic component does.
\end{Corollary} 

In the case $f$ is a $C^{1+\alpha}$-diffeomorphism, $\alpha>0$,
Pesin and Sinai ~\cite{PS} have introduced a class of invariant measures supported on unstable laminations
(which they called Gibbs u-states):
these are measures whose disintegrations along the unstable plaques of a laminated box of the unstable lamination
$\cF^u$ are absolutely continuous with respect to the Lebesgue measure along the plaques,
see also~\cite[Chapter 11]{BDV}.
The set $\mathcal{M}_{\rm u}$ is included in this class of measures:

\begin{Theorem}[Theorem 3.4 in~\cite{L}]\label{t.ledrappier}
Let $f$ be a $C^{1+\alpha}$ diffeomorphism, $\alpha>0$.
Then for any measure $\mu\in \mathcal{M}_{\rm u}$, the disintegrations along the unstable leaves
are absolutely continuous with respect to the Lebesgue measure.
\end{Theorem}

\begin{Remark}
The converse property also holds for any $C^1$-diffeomorphism (but we will not use that property).
This is a consequence of our Theorem~\ref{thm-III-bis} below.
For that reason we prefer to define the \emph{Gibbs u-states} as the measures in the class $\mathcal{M}_u$,
which is also adapted to $C^1$-diffeomorphisms.
\end{Remark}

\subsection{Unstable density basis}
The notion of Lebesgue density points does not behaves well under iterations.
Pugh and Shub ~\cite{PuSh2} have introduced a notion of unstable density point
inside the leaves of a globally partially hyperbolic diffeomorphism,
and from then have defined Julienne density points inside the manifold.
We here extend unstable density points inside an attracting neighborhood of
a partially hyperbolic attracting set.

Throughout this section, $\La$ is an invariant set endowed with
a partially hyperbolic splitting $T_\La M=E^{cs}\oplus E^{uu}$ for a $C^{1+\alpha}$ diffeomorphism $f$ of
a compact manifold $M$, $\alpha>0$ and $\mathcal{C}^{u}$ is an unstable cone field on a neighborhood $U$
of $\Lambda$.

A \emph{u-disc} $D\subset U$ is an embedded $C^1$-disc of dimension $\dim(E^{uu})$ 
that is tangent to an unstable cone field $\mathcal{C}^{u}$.
For $\delta>0$ we denote by $B_{D}(z,\delta)$ the closed $\delta$-ball centered at $z$ for the metric induced on $D$.

We now fix $\delta>0$ arbitrarily and
for any $z\in D\cap \bigcap_{n\geq 0}f^{-n}(U)$ and $n\in\mathbb{N}$ we define
$$B_{D,n}(z):=f^{-n}(B_{f^n(D)}(f^n(z),\delta)).$$
\begin{Theorem}~\label{thm.u-density-point}
Let $\La$ be an invariant set endowed with
a partially hyperbolic splitting $T_\La M=E^{cs}\oplus E^{uu}$ for a $C^{1+\alpha}$ diffeomorphism $f$, $\alpha>0$.
Let $\cC^u$ be an unstable cone field on a neighborhood $U$
and   $D$ be a $C^{1+\alpha}$ u-disc in the basin of $\Lambda$.

The collection $\{B_{D,n}(z)\}_{n\in\mathbb{N},z\in D}$ is a density basis of the u-disc $D$:
if $A\subset D\cap \bigcap_{n\geq 0}f^{-n}(U)$ is a measurable set with positive Lebesgue measure, then for Lebesgue almost
every $z\in A$,
$$\lim_{n\rightarrow\infty}\frac{\Leb_D(A\cap B_{D,n}(z))}{\Leb_D(B_{D,n}(z))}=1.$$
Such a point $z$ is called an \emph{unstable density point of $A$ in $D$}.
\end{Theorem}
The proof follows~\cite[Theorem 3.1]{PuSh2}. For completeness, we present it here.
\proof
Since $E^{uu}$ is uniformly expanded, there exist $N\geq 1$ and $\tau>1$
such that $\|Df^N(v)\|\geq \tau$ for any unit vector $v\in \cC^{u}$ at a point $x\in U\cap\dots\cap f^{-N}(U)$.
\begin{Lemma}\label{l.density-basis}
	\begin{enumerate}
		\item $\Leb_D(B_{D,n}(z))$ tends to zero as $n\to \infty$.
		\item  For any $m\in\mathbb{N}$, there exists a constant $K>1$ such that 
		$\Leb_D(B_{D,n}(z))\leq K\; {\Leb_D(B_{D,n+m}(z))}$.
		\item There exists $\ell\in\mathbb{N}$ such that
		$B^u_{D,n+\ell}(z_1)\cap B^u_{D,n+\ell}(z_2)\neq\emptyset\Rightarrow B^u_{D,n+\ell}(z_1)\subset B^u_{D,n}(z_2)$. 
	\end{enumerate}
\end{Lemma}
\begin{proof}
The first and the third items follow from the expansion along $E^{uu}$.
The second uses a distortion argument: for any $0\leq k\leq n$ and any two points
$x,y\in B_{f^n(D)}(f^n(z),\delta)$, the tangent spaces at $f^{-k}(x), f^{-k}(y)$
to $Tf^{-k}(B_{f^n(D)}(f^n(z),\delta))$ get exponentially close as $k$ gets larger.
Since $f$ is $C^{1+\alpha}$, the determinants $|\det Df|_{B_{f^n(D)}(f^n(z),\delta)}|$
at $f^{-k}(x)$ and $f^{-k}(y)$ is exponentially close in $k$, which concludes.
\end{proof}

Let $A\subset D\cap \bigcap_{n\geq 0}f^{-n}(U)$ be a measurable subset with positive Lebesgue measure in $D$.
For any $\rho\in(0,1)$, we denote  
$$A_\rho=\bigg\{z\in A: \liminf_{n\rightarrow\infty} \frac{\Leb_D(A\cap B_{D,n}(z))}{\Leb_D(B_{D,n}(z))}<\rho\bigg\}.$$
One only needs to  show that $A_\rho$ has zero Lebesgue measure.

For $\e>0$, we take an open neighborhood $U$ of $A_\rho$ such that $\Leb_D(U)<(1+\e)\Leb_D(A_\rho)$
and we consider a covering of $A_\rho$ given by 
$$\cV=\bigg\{ B_{D,n}(z)\subset U: \textrm{$z\in A_\rho$ and $\frac{\Leb_D(A\cap B_{D,n}(z))}{\Leb_D(B_{D,n}(z))}<\rho$}\bigg\}.$$
We then build inductively a sequence $\{V_i\}$ of pairwise disjoint sets in $\cV$ as follows.
Let us assume that the $V_j$ for $j<i$ have been chosen. Since they are closed sets,
and the diameter of the $B_{D,n}(z)$ tend to zero, for any point $z\in A_\rho\setminus \cup_{j<i} V_j$ there is an integer $n(z)=n_{i}(z)$ such that $B_{D,n(z)}(z)\in\cV$ is disjoint from the $V_j$ and we may choose the smallest integer $n(z)$ with this property.
We choose $z_i\in A_\rho\setminus \cup_{j<i} V_j$ which minimizes $n(z_i)$ and we take $V_i=B_{D,n(z_i)}(z_i)$.
\begin{Claim} The set $\tilde{A}_\rho:=A_\rho\setminus \cup_{i\in\mathbb{N}} V_i$ has zero Lebesgue measure in $D$.
\end{Claim}
\proof
For any integer $i$ and $z\in\tilde{A}_\rho$,
we consider the integer $n_i(z)$ introduced during the construction of $V_i$.
The set $B_{D,n_i(z)}(z)$ does not belong to $\{V_i\}$.

Note that by definition there exists $V_k=B_{D,n_k(z_k)}(z_k)$ with $k>i$ such that
$$V_k\cap B_{D,n_i(z)}(z)\neq \emptyset \text{ and } n_k(z_k)\leq n_i(z).$$
By the third item in Lemma~\ref{l.density-basis},
$$B_{D,n_i(z)}(z)\subset B_{D,n_k(z_k)-\ell}(z_k).$$
For any $k>\ell$, let us denote  $\tilde{V}_k=B_{D,n_k(z_k)-\ell}(z_k)$.  
We have proved that for any integer $i$, 
$$\tilde{A}_\rho\subset \bigcup_{k=i}^\infty\tilde{V}_k.$$

By the second item of Lemma~\ref{l.density-basis}, there exists a constant $K>1$ such that 
$$\Leb_D(\tilde{V}_k)<K\Leb_D(V_k).$$
Since the $V_k$ are pairwise disjoint, $\sum_{k\in\mathbb{N}}\Leb_D(V_k)$ converges.
For each $i$, one has 
$$\Leb_D(\tilde{A}_\rho)\leq\sum_{k=i}^\infty\Leb_D(\tilde{V}_k)\leq K\cdot\sum_{k=i}^\infty\Leb_D(V_k),$$
which implies that $\tilde{A}_\rho$ has zero Lebesgue measure. 
\endproof
By the Claim above, one has the estimate
\begin{eqnarray*}
	\Leb_D(A_\rho)&=&\sum_{i\in\mathbb{N}}\Leb_D(V_i\cap A_\rho)\leq\sum_{i\in\mathbb{N}}\Leb_D(V_i\cap A)\\
	&\leq& \rho \cdot \sum_{i\in\mathbb{N}}\Leb_D(V_i) \leq \rho \cdot\Leb_D(U)\\
	&\leq& \rho\cdot (1+\varepsilon)\Leb_D(A_\rho).
\end{eqnarray*}	
By the arbitrariness of $\e$ and the fact that $\rho<1$, one has $\Leb_D(A_\rho)=0$.
\endproof

\section{Measurable partitions associated to an unstable lamination}~\label{s.measurable-partitions-to-unstable}
The aim of this section is to construct finite partitions which allow to approximate the entropy along an unstable lamination.
One can find such constructions in~\cite{HHW,Yan16} for global partially hyperbolic diffeomorphisms:
\cite[Section 4]{Yan16} provide finite partitions which satisfy the first two items in the theorem below;
in~\cite[Propositions 2.12 and 2.13]{HHW} ~\cite{HHW} the entropy along an unstable lamination is
approached by the entropy of finite partitions that are conditioned by measurable partitions.
One of the novelty of the next theorem is the third item, which will crucial in Section~\ref{s.volume-estimate}.

\begin{Theorem}~\label{thm.switch-to-finite}
Let $f$ be a $C^1$-diffeomorphism of a compact manifold, $\Lambda$
be a $u$-laminated set with an unstable cone field $\cC^{u}$ on a neighborhood $U$. There is $r_0>0$ with the following properties.
For any $\mu\in \cM_{\rm inv}(\La, f)$ and any $\e,\rho>0$, there exist $\eta_0>0$, an integer $m_0$,  and  two finite measurable partitions $\alpha\prec\beta$ of $M$  such that 
  \begin{itemize}
  	\item $\diam(\beta)\leq\diam(\alpha)\leq \rho$;
  	\item any (not necessarily invariant) probability measure $\nu$ such that $\ud(\nu,\mu)<\eta_0$ satisfies
	$$\bigg|\frac{1}{m_0}H_\nu\left(\bigvee_{i=1}^{m_0}f^{-i}(\alpha)\big|\beta\right)-h_\mu(f,\cF^{u})\bigg|<\e;$$
	\item for any $\delta>0$, there exist an open set $V$ and an integer $N\geq 1$ such that 
	\begin{itemize}
		\item[--] $\mu(V)>1-\delta;$
		\item for any $x\in V$ and any disc $D$ tangent to $Df^N(\cC^{u})$ with $x\in D$ and ${\rm Diam}(D)<r_0$,
		$$\alpha(x)\cap D=\beta(x)\cap D.$$
		\end{itemize}
	\end{itemize}
\end{Theorem}

The proof of Theorem~\ref{thm.switch-to-finite} occupies the next  three  sections.

\subsection{Measurable partitions $\mu$-subordinate to an unstable lamination}\label{s.partition-to-ubstable-foliation}
In the following, we will construct a measurable partition $\mu$-subordinate to the strong unstable lamination.
A similar construction is done in~\cite{Yan16} in the case of global partially hyperbolic diffeomorphisms.
 
\paragraph{Transverse sections.}
The unstable cone field $\cC^{u}$ is defined on a small neighborhood $U$ of $\Lambda$.
The compactness of $\La$ and the transversality between $E^{cs},E^{uu}$ give:
\begin{Lemma}\label{l.lamination-box}
There exist $\rho_0>0$ and a family of compact discs $(\Sigma_x)_{x\in \Lambda}$ in $U$ satisfying:
	\begin{itemize}
	\item the disc $\Sigma_x$ has dimension $\dim(E^{cs})$, is centered at $x$, and has radius larger than $\rho_0;$ 
			\item $\Sigma_x$ is transverse to  $\cC^{u}$;
				\item for any $x,y\in \La$ with $d(x,y)<\rho_0$,
				$\cF^{u}_{2\rho_0}(y)$ intersects $\Sigma_x$ at a unique point, in the interior of $\Sigma_x$; in particular the set $ \cup_{y\in\Sigma_x}\cF^{u}_{2\rho_0}(y)$ contains an open $\rho_0$-neighborhood of $x$ in $\La$. 
	\end{itemize}
\end{Lemma}

The set $\La$ is covered by balls $B_1,\cdots,B_k$
of radius $\rho_0$ centered at points $x_1,\cdots,x_k\in\La$. Set $\Sigma_i=\Sigma_{x_i}$ for simplicity.

\paragraph{The choice of $r_0$.} We denote by $r_0$ the Lebesgue number of the covering $\{B_1,\cdots,B_k\}$.

\paragraph{A finite partition $\cA$.}
Let $\lambda=\sup_{x\in\La}\|Df^{-1}|_{E^{uu}(x)}\|<1$. We apply
the following lemma.
\begin{Lemma}[Lemma 3.1.2 in~\cite{Yan16} and Proposition 3.2 in~\cite{LS}]\label{l.partition}
For any $0<\lambda<1$ and $\rho>0$, there is a finite measurable partition $\cA$ of $M$ such that
	
	$$\diam(\cA)<\rho\textrm{ and } \sum_{i\in\mathbb{N}}\mu(B_{\lambda^i}(\partial{\cA}))<+\infty,$$
	where $B_{\lambda^i}(\partial{\cA})$ denotes the $\lambda^i$-neighborhood of the boundary  $\partial{\cA}$ of the partition $\cA$. 
	
\end{Lemma}
 
For any $\rho>0$, one gets a finite measurable partition $\cA$ of the manifold  $M$ such that
\begin{itemize}
	\item $\diam(\cA)<\min\{\rho,r_0/3,1\}$;
	\item $\sum_{i\in\mathbb{N}}\mu(B_{\lambda^i}(\partial{\cA}))<+\infty$; in particular, $\mu(\partial\cA)=0.$
\end{itemize}
Then $\cA$ induces a finite partition $\tilde{\cA}$ of $\La$. By construction, there exists an indexing  map $\cI: \tilde{\cA}\mapsto\{1,\cdots,k\}$ such that the $2r_0/3$-neighborhood of  each element $ A\in\tilde{\cA}$ in $\La$ is contained in the ball $B_{\cI(A)}$. From now on, the indexing map $\cI$ is fixed.

\paragraph{The partition $\cA^u$.}
Each point $x\in \La$ belongs to the set $B_{\cI(A(x))}$ and there exists a unique point $y\in \Sigma_{\cI(A(x))}$
such that $x\in \cF^{u}_{2\rho_0}(y)$; we set $\cA^u(x)={A}(x)\cap \cF_{2\rho_0}^{u}(y)$.
This defines a measurable partition $\cA^u$ on $\La$.
\medskip

We note that the assumption of Lemma~\ref{l.alternative} is satisfied.
\begin{Lemma}\label{l.first-finite} $H_{\mu}(\cA^u|f(\cA^u))<\infty.$
\end{Lemma}
\proof By definition, one has that
$H_{\mu}(\cA^u|f(\cA^u))=\int H_{\mu^{f(\cA^u)(x)}}(\cA^u|f(\cA^u)(x))\ud\mu(x)$.
By  definition, $\cA^u$ and $\tilde \cA$ induce on each element $f(\cA^u)(x)\in f(\cA^u)$ the same partition, which is a finite partition. Hence, one has that  $H_{\mu^{f(\cA^u)(x)}}(\cA^u|f(\cA^u)(x))\leq \log{\#\tilde{\cA}}$ which implies
$H_{\mu}(\cA^u|f(\cA^u))\leq\log{\#\tilde{\cA}}$.
\endproof
\medskip

We obtain a partition $\mu$-subordinate to the unstable lamination.
\begin{Lemma}\label{l.increasing partition}
	  $\bigvee_{j=0}^\infty f^j(\cA^u)$ is an increasing partition  $\mu$-subordinate to the lamination $\cF^u$.
\end{Lemma}
\proof
Firstly, notice that the measurable partition $\bigvee_{j=0}^\infty f^j(\cA^u)$ is an increasing partition.

For $\mu$ a.e.  $x\in\La$, one claims that there exists an integer $n(x)\in\mathbb{N}$ such that
$$\bigvee_{j=0}^\infty f^j(\cA^u)(x)=\bigvee_{j=0}^{n(x)} f^j(\cA^u)(x).$$
Since $\mu$ is an invariant measure supported on $\La$, one has
$$\sum_{j=0}^\infty\mu(f^j(B_{\lambda^j}(\partial(\cA))))=\sum_{j=0}^\infty\mu(B_{\lambda^j}(\partial(\cA)))<\infty.$$
Hence, for $\mu$-a.e. $x$, there is $n(x)$ such that$f^{-j}(x)\notin B_{\lambda^j}(\partial(\cA))$ for any $j\geq n(x)$.
Let us assume that there exists $m\geq n(x)$ such that
$$\bigvee_{j=0}^{m+1} f^j(\cA^u)(x)\subsetneq\bigvee_{j=0}^{m} f^j(\cA^u)(x).$$
Since $\diam(\cA)<1$, the diameter of $\bigvee_{j=0}^{m} f^j(\cA^u)(x)$ is smaller than $\lambda^{-m}$,
which implies $f^{-m-1}(x)\in B_{\lambda^{m+1}}(\partial\cA) $, a contradiction. The claim follows.

By the fact that $\mu(\partial(\cA))=0$, for $\mu$-a.e. $x$, the element $\cA^u(x)$ contains an open set  $\cF_{r(x)}^u(x)$ for some $r(x)>0$. Therefore, $\bigvee_{j=0}^{n(x)} f^j(\cA^u)(x)$ contains a neighborhood of $x$ in $\cF^u(x)$.
\endproof
\medskip

In particular, Lemmas~\ref{l.alternative}, \ref{l.first-finite} and~\ref{l.increasing partition},
together with Definition~\ref{def.partial-entropy} give:
\begin{Corollary}~\label{c.11}
	$h_\mu(f,\cF^{u})=\inf \frac{1}{m}H_\mu(\bigvee_{j=1}^m f^{-j}(\cA^u)|\cA^u)=\underset{m\rightarrow\infty}\liminf\frac{1}{m}H_\mu(\bigvee_{j=1}^m f^{-j}(\cA^u)|\cA^u).$
\end{Corollary}
\medskip

One important property of the measurable partition $\bigvee_{i=0}^l f^{-i}\cA^u$ is the following:
\begin{Lemma}\label{l.order equivalent}
	For $\mu$-a.e. $x\in M$ and any integer $m\in\mathbb{N}$,
	$$\bigvee_{j=0}^m f^{-j}(\cA^u)(x)=\bigvee_{j=0}^m f^{-j}(\cA)(x)\cap \cA^u(x).$$
\end{Lemma}
\proof
Since $\cA\prec\cA^u$, one gets the inclusion
$\bigvee_{i=0}^m f^{-i}(\cA^u)(x)\subset\bigvee_{i=0}^m f^{-i}(\cA)(x)\cap \cA^u(x)$.
One proves the other side by induction. The case $m=0$, is obvious.

Let us assume that $\bigvee_{i=0}^m f^{-i}(\cA)(x)\cap \cA^u(x)\subset \bigvee_{i=0}^m f^{-i}(\cA^u)(x)$.
Consider any point $y$ in $\bigvee_{i=0}^{m+1} f^{-i}(\cA)(x)\cap \cA^u(x)$.
The induction assumption implies $f^l(y)\in\cA^u(f^l(x))$ for all $l\in\{0,\cdots,m\}$;
one thus has $f^{m+1}(y)\in f(\cA^u(f^{m}(x)))\cap \cA(f^{m+1}(x))$.
By the definition of $\cA^u$,  the point $f^{m+1}(y)$ belongs to
$f(\cF^u_{2\rho_0}(z))\cap \cA(f^{m+1}(x))$ for some point $z\in \Lambda$.
Since $\rho_0$ is small, there exists $\zeta\in \Si_{\cI(A(f^{m+1}(x)))}$
such that $f(\cF^u_{2\rho_0}(z)\cap \cA(f^{m+1}(x))= \cF^u_{2\rho_0}(\zeta))\cap \cA(f^{m+1}(x))$.
Hence $f^{m+1}(y)\in\cA^u(f^{m+1}(x))$.
\endproof

\subsection{Finite partitions approaching $\cA^u$}
We continue with the constructions of the previous subsection.

\begin{Proposition}~\label{p.finite-partition-infinite-partition}
	There exist finite measurable partitions $(\cA^u_l)_{l\in\mathbb{N}}$ of $M$ such that 
	\begin{itemize}
			\item $\mu(\partial{\cA^u_l})=0$ for any $l\in\mathbb{N}$;
		\item  for $\mu\: a.e. \: x\in M$, $\cA^u_{l+1}(x)\subset \cA^u_l(x)\subset \cA(x)$ and  $\cA^u(x)=\bigcap_{l\in\mathbb{N}}\cA^u_l(x)$. 
		\item for any $l\in\mathbb{N}$ and any $\delta>0$, there exist an open set $V$ and  an integer $N$ such that 
		\begin{itemize}
			\item[--] $\mu(V)>1-\delta$;
			\item[--] for any $x\in V$ and any disc $D$ tangent to $Df^N(\cC^{u})$ with $x\in D$
			and $\diam(D)\leq r_0$,
			$$D\cap \cA^u_l(x)=D\cap \cA(x).$$
			 \end{itemize}
		\end{itemize}
	\end{Proposition}
\proof Let $\{B_i\}_{i\in\{1,\cdots,k \}}$ be the open cover of  $\La$   given in the previous section and $\Sigma_i$ be the associated transverse discs. 
For each $\Sigma_i$,  the collection of local unstable manifolds$\{\cF^{u}_{2\rho_0}(y)\}_{y\in\Sigma_i\cap \La}$ defines a measurable partition of $\cup_{y\in\Sigma_i\cap \La}\cF^{u}_{2\rho_0}(y)$, and we denote by $\mu_i$ the projection on $\Sigma_i$ of
the measure $\mu$ restricted to $\cup_{y\in\Sigma_i\cap \La}\cF^{u}_{2\rho_0}(y)$.

For each $\Sigma_i$, there is a sequence of finite partitions $\cC_{i,1}\prec\cC_{i,2}\prec\cdots\prec \cC_{i,l}\prec\cdots$ such that
\begin{itemize}
	\item $\diam(\cC_{i,l})\stackrel{l\rightarrow\infty}{\longrightarrow}0$
	\item $\mu_i( \partial{\cC_{i,l}})=0$, where $\partial{\cC_{i,l}}$ denotes the boundary of partition $\cC_{i,l}$ in $\Sigma_i.$
\end{itemize}
Then we denote by $\tilde{\Sigma}_i$ the set $(\Lambda\cap\Sigma_i)\setminus \cup_l \partial{\cC_{i,l}}$.

Let us fix any $l$ and any $C\in \cC_{i,l}$.
For any $x\in \tilde{\Sigma}_i\cap C$, since $x$ is an interior point of $C$, there exists 
$r_{x,l}>0$ such that distance between $\cF^u_{2\rho_0}(x)$ and any other local leaf
$\cF^u_{2\rho_0}(y)$ with $y\in ( \tilde\Si_i\cap\Lambda)\setminus C$ is larger than $3r_{x,l}$.
We define the set $\tilde{C}$
which is the union of the $r_{x,l}$-neighborhood of the local leaf $\cF^u_{2\rho_0}(x)$ over $x\in \tilde{\Sigma}_i\cap C$.

By construction, $\tilde{C}\cap\Sigma_i$ is an open set in $C$, and
has full $\mu_i$-measure in $C$; in particular the boundary of  $\tilde{C}\cap\Si_i$ in $\Si_i$ has $\mu_i$-measure zero.
Moreover by the choice of the numbers $r_{x,l}$, the $\tilde C$'s for different $C\in \cC_{i,l}$ are pairwise disjoint.
The partition $\cP_{i,l}$ for $B_i$ given by 
$$\big\{\tilde{C}\big\}_{C\in\cC_{i,l}}\bigcup \big\{B_i\setminus\bigcup_{C\in\cC_{i,l}}\tilde{C}\big\},$$
is a finite measurable partition whose boundary in $\Si_i$ has $\mu_i$ measure zero.
One may also require the condition $r_{x,l+1}<r_{x,l}$ for each $x\in \tilde{\Sigma}_i$ and each $l$:
this gives $\cP_{i,l+1}\prec\cP_{i,l}$ modulo a set with zero $\mu$ measure.
Since $\diam(\cC_{i,l})\stackrel{l\rightarrow\infty}{\longrightarrow}0$,   
one has $\cap_{l}\cP_{i,l}(x)=\cF^{u}_{2\rho_0}(x)$ for $x\in \tilde{\Sigma}_i$.

For each $A\in\cA$ with $\cI(A)=i$, the finite partition $\cP_{i,l}$ induces a finite measurable partition for $A$, and this defines finite partitions $\cA^u_l$. The fact that $\cap_{l}\cP_{i,l}(x)=\cF^{u}_{2\rho_0}(x)$ for $x\in\tilde{\Sigma}_i$ implies that $\cA^u_l$ satisfies the second item.

Recall that for each $A\in\cA$ with $\cI(A)=i$,  one has $B_{2r_0/3}(A)\subset B_i$. For each $x\in A$, the boundary of the set $\cA^u_l(x)$ is contained in $\partial{A}$ and $\partial{\cP}_{i,l}(x)$. The fact that $B_{2r_0/3}(A)\subset B_i$ and $\mu$ is supported on $\La$ implies that up to modulo a set of zero $\mu$ measure, one has 
$$\partial(\cA^u_l(x))=\partial A \cup \{y\in B_i, \cF^u_{2\rho_0}(y) \cap\tilde{\Sigma}_i=\emptyset\}.$$
Since $\mu(\partial(A))=0$ and $\tilde{\Sigma}_i$ has full $\mu_i$ measure, one has $\mu(\partial(\cA^u_l(x)))=0$.

It remains to prove the last item. We fix an integer $l\in\mathbb{N}$ and $\delta>0$. 
\begin{Claim}
For each $A\in\cA$ with $\mu(A)>0$,	there is an open subset $V_A$ of $A$ and $N_A\geq 1$ such that 
	\begin{itemize}
		\item $\mu(V_A)> (1-\delta)\cdot \mu(A)$,
		\item for any $x\in V$ and any disc $D$ tangent to $Df^{N_A}(\cC^{u})$ with $x\in D$
			and $\diam(D)\leq r_0$,
		        $$D\cap \cA^u_l(x)=D\cap A.$$
		\end{itemize} 
\end{Claim}
\proof[Proof of Claim]
Let $A\in\cA$ with $\mu(A)>0$ and let $i=\cI(A)$. Since $\mu(\partial A)=0$, there is an open set  $A^\prime\subset A$ such that $\mu(A^\prime)=\mu(A)$. 

For each $C\in\cC_{i,l}$ with $\tilde{C}\cap A\neq\emptyset$, the open set $\tilde{C}$ intersects $C$ into an open subset $C^\prime$ of $C$ whose boundary has zero $\mu_i$-measure; then for $\delta>0$, there exist $r_\delta>0$ and a compact  subset $C^{\prime\prime}$ of $C^{\prime}$ such that 
\begin{itemize}
	\item for any $x\in C^{\prime\prime}\cap \tilde{\Si}_i$, the $2r_\delta$-neighborhood of $\cF^{u}_{2\rho_0}(x)$ is
	included in $\tilde{C}$;
	\item $\mu_i(C^{\prime\prime})> \mu_i(C^\prime)-\delta\cdot\frac{\mu(A)}{\#\cC_{i,l}}$.
	\end{itemize}
Now, we define $\hat{C}$ as the union of the $r_\delta$-neighborhood of $\cF^{u}_{2\rho_0}(x)$ over $x\in C^{\prime\prime}\cap \tilde{\Si}_i$. By definition, the closure of $\hat{C}$ is included in $\tilde{C}$.
Let $V_A$ be the union of $\hat{C}\cap A^\prime$ over all $C\in \cC_{i,l}$ with $\tilde{C}\cap A\neq\emptyset$. By the fact that $B_{2r_0/3}(A)\subset B_i$, one gets $\mu(V_A)> (1-\delta)\cdot\mu(A)$.
 
Any disc of radius less or equal to $r_0$, that is $C^1$-close to a leaf $\cF^{u}_{2\rho_0}(x)$ for
$x\in C^{\prime\prime}\cap \tilde{\Si}_i$ and having a point in $B_i$,
is contained in $\widehat C$. By compactness of $C^{\prime\prime}$,
one deduces that if one chooses an integer $N_C\geq 1$
large enough and $r_\delta>0$ small enough, then the following property holds:
for any $n\geq N_C$, any disc $D$ intersecting $V_A=\hat{C}\cap A'$ with diameter $\leq r_0$ and tangent to $Df^{n}(\cC^{u})$ is contained in $\tilde{C}$.
By definition of $\cA^u_l$, for $x\in D$ one gets $D\cap \cA^u_l(x)=D\cap A$.

 Since $\cC_{i,l}$ is finite, one concludes by taking $N_A=\max N_C$ over $C\in\cC_{i,l}$  with $\tilde{C}\cap A\neq\emptyset$.
\endproof
For each $A\in\cA$ with $\mu(A)=0$, we define $V_A=\emptyset$. 
We take $V=\cup_{A\in\cA} V_A$ and $N\gg \max N_A$. By the Claim above, the open set $V$ satisfies the required properties.
\endproof

 \subsection{Proof of Theorem~\ref{thm.switch-to-finite}}
From Corollary~\ref{c.11}, the measurable partition $\cA^u$ satisfies
$$h_\mu(f,\cF^u)=\liminf_{m\rightarrow\infty}\frac{1}{m}H_\mu\big(\bigvee_{j=1}^m f^{-j}(\cA^u)|\cA^u\big).$$
Thus, for any $\varepsilon>0$, there exists an integer $m_0>0$ such that 
$$\big|\frac{1}{m_0}H_\mu\big(\bigvee_{i=1}^{m_0}f^{-i}(\cA^u)|\cA^u\big)-h_\mu(f,\cF^u)\big|\leq  \frac{\e}{3}.$$
By Lemma~\ref{l.order equivalent}, we have 
$\bigvee_{j=0}^{m_0} f^{-j}(\cA^u)=\bigvee_{j=0}^{m_0} f^{-j}(\cA)\vee\cA^u$ (modulo a set with $\mu$-measure zero).
Hence
$$ H_\mu\big(\bigvee_{i=1}^{m_0}f^{-i}(\cA^u)|\cA^u \big)= H_\mu\big(\bigvee_{i=1}^{m_0}f^{-i}(\cA)|\cA^u\big).$$
From Proposition~\ref{p.finite-partition-infinite-partition},
the sequence of finite measurable partitions $(\cA^u_l)_{l\in\mathbb{N}}$
satisfies $\cA^u_l\prec\cA^u_{l+1}\prec\cA^u$ and $\cA^u=\bigvee\cA^u_l$
(modulo a set with $\mu$-measure zero).
From the second item of Lemma~\ref{l.updown}, there exists an integer $l_0$ such that 
$$\big|\frac{1}{m_0}H_\mu\big(\bigvee_{i=1}^{m_0}f^{-i}(\cA)|\cA^u\big)-\frac{1}{m_0}H_\mu\big(\bigvee_{i=1}^{m_0}f^{-i}(\cA)|\cA^u_{l_0}\big)\big|<\frac{\e}{3}.$$
As a consequence, one has 
$$\big|\frac{1}{m_0}H_\mu\big(\bigvee_{i=1}^{m_0}f^{-i}(\cA)|\cA^u_{l_0}\big)-h_\mu(f,\cF^u)\big|<\frac{2\e}{3}.$$
By construction, one has $\mu(\partial(\cA^u_{l_0}))=\mu(\partial(\cA))=0$. Thus
there exists $\eta_0>0$ such that for any probability measure $\nu$
with $d(\mu,\nu)<\eta_0$, one has
	$$\big|\frac{1}{m_0}H_\nu\big(\bigvee_{i=1}^{m_0}f^{-i}(\cA)|\cA_{l_0}^u\big)-\frac{1}{m_0}H_\mu\big(\bigvee_{i=1}^{m_0}f^{-i}(\cA)|\cA_{l_0}^u\big)\big|<\frac{\e}{3}.$$
 To summarize, for any probability measure $\nu\in B_{\eta_0}(\mu)$, one has 
 	$$\big|\frac{1}{m_0}H_\nu\big(\bigvee_{i=1}^{m_0}f^{-i}(\cA)|\cA_{l_0}^u\big)-h_\mu(f,\cF^{u})\big|<\e.$$
 Now, one only needs to take $\alpha=\cA$ and $\beta=\cA^u_{l_0}.$
 By the choice of $\cA$ in Section~\ref{s.partition-to-ubstable-foliation} we have $\diam(\alpha)<\rho$
 and by construction $\alpha\prec \beta$.
 
 For any $\delta$, the existence of $V$ and $N$
 as in the last property of Theorem~\ref{thm.switch-to-finite} is guaranteed by the third item of Proposition~\ref{p.finite-partition-infinite-partition} for the partition $\cA^u_{l_0}$. This ends the proof of Theorem~\ref{thm.switch-to-finite}.
\qed

\section{Volume estimate for convergent sets of invariant measures}~\label{s.volume-estimate}

Given an invariant measure $\mu$ of $f\in\diff^1(M)$, we define for any $n\geq 1$ and $\eta>0$
the \emph{$(n,\eta)$-convergent set}:
$$C_n(\mu,\eta):=\bigg\{x\in M: \ud\bigg(\frac{1}{n}\sum_{i=0}^{n-1}\delta_{f^i(x)},\mu\bigg)<\eta \bigg\}.$$

The aim of this section is to prove:
\begin{theoremalph}~\label{thm.volume-control-by-entropy}
Let $f$ be a $C^1$-diffeomorphism of a compact manifold and $\Lambda$ be a $u$-laminated set.
Then, there exist an unstable cone field $\cC^{u}$ on a neighborhood $U$ of $\La$  and $r_0>0$ with the following property:
	for any $\mu\in\cM_{\mathrm{inv}}(\La,f)$ and $\e>0$, there exist $\eta,c>0$ such that
	for each compact disc $D\subset U$ tangent to $\cC^{u}$ with $\diam(D)<r_0$ and each $n\in\mathbb{N}$, one has 
	$$\Leb_D\big(  C_{n}(\mu,\eta)\cap D\cap \bigcap_{i=0}^{n-1}f^{-i}(U)\big)<c\cdot \exp\bigg(n\bigg(h_\mu(f,\cF^u)-\int\log{|\det(Df|_{E^{uu}})|}\ud\mu+\e\bigg)\bigg).$$
	\end{theoremalph}

%
%

\subsection{Preliminary choices}
\paragraph{\bf Choice of $\cC^{u}, U,r_0$.} From the partially hyperbolicity of $\La$,
there exist a neighborhood $U$ of $\La$, an unstable cone field $\cC^{u}$ defined on $U$, $\lambda\in(0,1)$
and $N\geq 1$ such that 
\begin{itemize}
	\item $Df^N(\cC^{u}(x))\subset \cC^{u}(f^N(x))$, for  any $x\in {U}\cap f^{-1}({U})\cap\cdots\cap f^{-N}( U)$;
	\item $\|Df^{-N}(v)\|\leq \lambda$ for any $x\in {U}$ and any unit vector $v\in\cC^{u}(x)$;
\item
the partially hyperbolic splitting $E^{cs}\oplus E^{uu}$ extends to the maximal invariant set in $U$.
\end{itemize}
We choose a continuous extension $\psi\colon M\to \mathbb{R}$ of the map $x\mapsto-\log{|\det(Df|_{E^{uu}})(x)|}$
defined on the maximal invariant set in $U$.
We also fix a number $r_0>0$ which satisfies Theorem~\ref{thm.switch-to-finite}.

\paragraph{Cone field $\cC^{u}_\varepsilon$.}
Let us fix $\e>0$. There exist $N_\e,L_\e\geq 1$ and $\rho>0$ such that
\begin{itemize}
	\item for any $x,y\in M$ with $\ud(x,y)<\rho$, one has $|\psi(x)-\psi(y)|<\frac{\e}{8};$
\item 
the cone field $\cC^{u}_\e:=Df^{N_\e}\cC^u$ defined on
$U_\e:=\cap_{i=0}^{N_\e-1}f^k(U)$ satisfies:
\begin{itemize}
	\item[--]  for any  disc $D\subset U_\e$ tangent to $\cC^{u}_\e$ and any $x\in D$,
$\big| \log{|\det(Df|_{T_xD_x})|}+\psi(x)\big|<\frac{\varepsilon}{8};$
\item[--] for any compact disc $D\subset U$ with diameter smaller than $r_0$ and tangent to $\cC^{u}$,
the set $f^{N_\e}(D)\cap \bigcap_{i=0}^{N_\e-1}f^i(U)$ is contained in at most $L_\e$ discs tangent to $\cC^{u}_\e$
of diameter $r_0$.
\end{itemize}
\end{itemize}

We fix an invariant measure $\mu$ on $\La$.
From the previous properties, one only needs to prove the Theorem~\ref{thm.volume-control-by-entropy} for discs contained in $U_\e$, tangent to $\cC^{u}_\e$ and with diameter bounded by $r_0.$

\subsection{Volume estimate through pressure}
Let us fix a disc $D$ tangent to $\cC^{u}_\e$ with diameter smaller than $r_0$ and some integer $n\geq 0$.
A set $X$ is \emph{$(n,\rho)$-separated}
if any $x,y\in X$ satisfy $\ud(f^k(x),f^k(y))>\rho$ for some $0\leq k< n$.
For each $x\in D$, we denote by $B_n(x,\rho)$ the \emph{$(n,\rho)$-Bowen ball} in $D$ and centered at $x$, that is,
$$B_n(x,\rho)=\bigcap_{i=0}^{n-1} f^{-i}(B(f^i(x),\rho))\cap D.$$
For any   $n\in\mathbb{N}$, $\eta>0$ and $\rho>0$, let $X_{n,\rho}$ be a $(n,\rho)$-separated set with maximal cardinal of
$$C_{n}(\mu,\eta)\cap D\cap\bigcap_{i=0}^{n-1}f^{-i}(U).$$
We consider the probability measures:
$$\nu_n:=\frac{1}{\# X_{n,\rho}}\sum_{x\in X_{n,\rho}}\delta_x,
\quad\quad \mu_n:=\frac{1}{n}\sum_{i=0}^{n-1}f_*^i\nu_n
=\frac{1}{\# X_{n,\rho}} \sum_{x\in X_{n,\rho}} 
\frac{1}{n}\sum_{i=0}^{n-1}\delta_{f^i(x)}.$$
\begin{Remark} The definition of $ X_{n,\rho}$ and of the convexity of the ball of radius $\eta$ centered at $\mu$ in the space of probability measures gives $d(\mu_n,\mu)<\eta$.
\end{Remark}
\medskip

The volume of $C_{n}(\mu,\eta)$ is estimated as follows.

\begin{Proposition}\label{p.volume-through-separating-set}
	There exist $c_\e>0,\eta_1>0$ (which only depend on $\e$) such that for any $0<\eta<\eta_1$,
	and for any finite measurable partitions $\cP_0,\cdots,\cP_{n-1}$ of $M$ with  diameters smaller than $\rho$,
		$$\Leb_D\big(C_{n}(\mu,\eta)\cap D\cap\bigcap_{i=0}^{n-1}f^{-i}(U)\big)\leq
		 c_\e\cdot	\exp\big({\textstyle{\frac{n\e}{2}}}+n\int\psi\ud\mu+H_{\nu_n}\big(\bigvee_{i=0}^{n-1}f^{-i}(\cP_i)\big)\big).$$
\end{Proposition}
\begin{proof}
By the choice of $ X_{n,\rho},$ one has 
\begin{eqnarray*}
	\Leb_D\big(C_{n}(\mu,\eta)\cap D\cap\bigcap_{i=0}^{n-1}f^{-i}(U)\big)&\le&\sum_{x\in X_{n,\rho}}\Leb_D\big(C_{n}(\mu,\eta)\cap B_{n}(x,\rho)\cap\bigcap_{i=0}^{n-1}f^{-i}(U)\big).
\end{eqnarray*}
By the definition of $U_\e$, for any $0\leq i<n$
the point $f^i(y)$ is contained in $U_\e$ and $f^i(D)\cap \bigcap_{j=0}^{i}f^j(U)$ is tangent to the cone field $\cC^{u}_\e$.
By the choice of $\rho$, for $x,y$ in a same $(n,\rho)$-Bowen ball of $D$,
\begin{eqnarray*}
	\big|\log|\det(Df^{-n}|_{T_{f^n(y)} f^n(D)})|&-&\log|\det(Df^{-n}|_{T_{f^n(x)} f^n(D)})|\big|\\
	&\leq&\sum_{i=1}^{n}	\big|\log|\det(Df^{-1}|_{T_{f^i(y)} f^i(D)})|-\log|\det(Df^{-1}|_{T_{f^i(x)} f^i(D)})|\big|
	\\
	&\leq& \sum_{i=1}^{n}\big(|\psi(f^i(x))-\psi(f^i(y))|+{\textstyle \frac{\e}{4}}\big)\;\leq\;
	n\cdot{\textstyle \frac{3\e}{8}}.
\end{eqnarray*}
We denote $S_n\psi(z):=\sum_{i=0}^{n-1}\psi(f^i(z))$ for $z\in M$.
Then,
\begin{eqnarray*}
	\Leb_D\big(B_{n}(x,\rho)\cap\bigcap_{i=0}^{n-1}f^{-i}(U)\big)
	&\leq&   \int_{f^{n}\big(B_n(x,\rho)\cap\bigcap_{i=0}^{n-1}f^{-i}(U)\big)}|\det(Df^{-n}|_{T_y f^n(D)})|\ud\Leb_{f^{n}(D)}(y)\\
	&\le& c_\e\cdot e^{\frac{3n\e}{8}}\cdot e^{S_n\psi(x)},
\end{eqnarray*}
where $c_\e$ is an upper bound for the volume of the discs tangent to $\cC^{u}$ with diameter $\rho$.
   
For $\e>0$, there is $\eta_1>0$ such that for any probability measures $\nu_1,\nu_2$,  if $\ud(\nu_1,\nu_2)<\eta_1$, then $|\int\psi\ud\nu_1-\int\psi\ud\nu_2|<\frac{\e}{8}$. 
   Hence, for $\eta<\eta_1$ and $x\in X_{n,\rho}$, one has 
   $|\frac{1}{n} S_n\psi(x)-\int\psi\ud\mu|<\frac{\e}{8}.$
This gives the estimate
\begin{equation}\label{e1}
\Leb_D\big(C_{n}(\mu,\eta)\cap D\cap\bigcap_{i=0}^{n-1}f^{-i}(U)\big)\leq c_\e\cdot e^{\frac{n\e}{2}}\cdot e^{n\cdot\int\psi\ud\mu}\cdot\# X_{n,\rho}.
\end{equation}

Let $\cP_0,\cdots,\cP_{n-1}$ be finite measurable partitions with diameter smaller than $\rho$. By the choice of $ X_{n,\rho}$, 
each element of $\bigvee_{i=0}^{n-1}f^{-i}(\cP_i)$ contains at most one point of $ X_{n,\rho}$. Hence,
\begin{equation}
\begin{split}\label{e2}
H_{\nu_n}\big(\bigvee_{i=0}^{n-1}f^{-i}(\cP_i)\big)&=\sum_{x\in X_{n,\rho}}-\nu_n\big(\bigvee_{i=0}^{n-1}f^{-i}(\cP_i)(x)\big)\cdot \log\nu_n\big(\bigvee_{i=0}^{n-1}f^{-i}(\cP_i)(x)\big)\\
&=\sum_{x\in  X_{n,\rho}}\frac{1}{\# X_{n,\rho}}\cdot
\log\# X_{n,\rho}\;=\;\log\# X_{n,\rho}.
\end{split}
\end{equation}

The relations~\eqref{e1} and~\eqref{e2} together give the required estimate.
\end{proof}

\subsection{Localization along unstable leaves}
Theorem~\ref{thm.switch-to-finite} associates to $\mu,\e/4,\rho$,
and gives a number $\eta_0>0$, two partitions
 $\alpha\prec\beta$  of $M$ and $m_0\in\NN$. 
For any $0\leq \ell<m_0<n$, let $\cP^\ell_0,\cdots,\cP^\ell_{n-1}$ be finite measurable partitions of $M$ such that 
  \[ \cP^\ell_i = \left\{
 \begin{array}{rl}
 \beta & \text{if $i=\ell+km_0$} ,\\
 \alpha & \text{otherwise}.
 \end{array} \right. \]
For these partitions, we have to estimate the quantity $H_{\nu_n}\big(\bigvee_{i=0}^{n-1}f^{-i}(\cP_i)\big)$
which appears in Proposition~\ref{p.volume-through-separating-set}.
By Corollary~\ref{c.H-condition} and the fact $\#\beta\ge\#\alpha$, one gets
\begin{eqnarray}\label{e.estimateH}
&H_{\nu_n}\big(\bigvee_{i=0}^{n-1}f^{-i}(\cP^\ell_i)\big)
	=H_{\nu_n}\big(\bigvee_{i=0}^{\ell-1}f^{-i}(\alpha)\vee f^{-\ell}(\beta)\big)\hspace{8cm}\nonumber\\
 &\quad\quad\quad\quad +\sum_{k=0}^{[\frac{n-\ell}{m_0}]-1}H_{f^{\ell+km_0}_*\nu_n}\big(\bigvee_{i=1}^{m_0-1}f^{-i}(\alpha)\vee f^{-m_0}(\beta)|\bigvee_{i=0}^{km_0+\ell}f^{km_0+\ell-i}(\cP^\ell_i)\big)\nonumber\\
 & \quad\quad\quad\quad\quad\quad\quad\quad +H_{\nu_n}\big(\bigvee_{i=\ell+[\frac{n-\ell}{m_0}]m_0+1}^{n-1}f^{-i}(\cP^\ell_{i})|\bigvee_{i=0}^{\ell+[\frac{n-\ell}{m_0}]m_0}f^{-i}(\cP^\ell_i)\big)
 \nonumber\\
&\leq 2m_0\cdot\log\#\beta+\sum_{k=0}^{[\frac{n-\ell}{m_0}]-1}H_{f^{\ell+km_0}_*\nu_n}\big(\bigvee_{i=1}^{m_0-1}f^{-i}(\alpha)\vee f^{-m_0}(\beta)|\bigvee_{i=0}^{km_0+\ell}f^{km_0+\ell-i}(\cP^\ell_i)\big).
\end{eqnarray}
 
The main estimate is given by the following lemma.
 
\begin{Lemma}~\label{l.error-estimate}
There exist $\eta_2>0$ and $ N_2\geq 1$ (independent from the choice of $D$)
such that for any $n\geq  N_2$, and assuming $d(\mu_n,\mu)<\eta_2$, we have
\begin{align}\label{e.alpha-beta}
\sum_{\ell=0}^{m_0-1}\sum_{k=0}^{[\frac{n-\ell}{m_0}]-1}H_{f^{\ell+km_0}_*\nu_n}&\bigg(\bigvee_{i=1}^{m_0-1}f^{-i}(\alpha)\vee f^{-m_0}(\beta)\bigg|\bigvee_{i=0}^{km_0+\ell}f^{km_0+\ell-i}(\cP_i^\ell)\bigg)\nonumber\\
&-
\sum_{\ell=0}^{m_0-1}\sum_{k=0}^{[\frac{n-\ell}{m_0}]-1}H_{f^{\ell+km_0}_*\nu_n}\bigg(\bigvee_{i=1}^{m_0}f^{-i}(\alpha)\bigg|\bigvee_{i=0}^{km_0+\ell}f^{km_0+\ell-i}(\cP^\ell_i)\bigg)\;\leq\; \frac{n\e}{4}.	
\end{align}
\end{Lemma}
\proof
The third item of Theorem~\ref{thm.switch-to-finite} for  $\delta=\frac{\e}{4\log\#\beta}$
gives an open  set $V_\e$ and $n_\e$ such that 
\begin{itemize}
	\item $\mu(V_\e)>1-\delta/4$;
	\item for any disc $\tilde{D}$ tangent to $Df^{n_\e}(\cC^{u})$ containing $x\in V_\e$ and of diameter $r_0$,
	$$\tilde{D}\cap \beta(x)=\tilde{D}\cap \alpha(x).$$
	\end{itemize}
There exists $\eta_2>0$ such that for any probability measure $\nu$
satisfying $\ud(\mu,\nu)<\eta_2$, one has $\nu(V_\e)>1-\delta/2$.
In particular if one assumes $\ud(\mu,\mu_n)<\eta_2$, one gets by Lemma~\ref{Lem:H-condition}:
\begin{align*}
&H_{f^{\ell+km_0}_*\nu_n}\big({\textstyle \bigvee_{i=1}^{m_0-1}f^{-i}(\alpha)\vee f^{-m_0}(\beta)|\bigvee_{i=0}^{km_0+\ell}f^{km_0+\ell-i}(\cP_i^\ell)}\big)\\
&={\textstyle H_{f^{\ell+km_0}_*\nu_n}\big({\textstyle\bigvee_{i=1}^{m_0-1}f^{-i}(\alpha)|\bigvee_{i=0}^{km_0+\ell}f^{km_0+\ell-i}(\cP_i^\ell)\big)}}\\
&\hspace{1cm} \textstyle{+H_{f^{\ell+km_0}_*\nu_n}\big({\textstyle f^{-m_0}(\beta)|\bigvee_{i=1}^{m_0-1}f^{-i}(\alpha)\vee\bigvee_{i=0}^{km_0+\ell}f^{km_0+\ell-i}(\cP_i^\ell)}\big)}
\\
&={\textstyle H_{f^{\ell+km_0}_*\nu_n}\big({\textstyle\bigvee_{i=1}^{m_0-1}f^{-i}(\alpha)|\bigvee_{i=0}^{km_0+\ell}f^{km_0+\ell-i}(\cP_i^\ell)}\big)+H_{f^{\ell+(k+1)m_0}_*\nu_n}\big({\textstyle\beta|\bigvee_{i=0}^{(k+1)m_0+\ell-1}f^{(k+1)m_0+\ell-i}(\cP_i^\ell)}\big)}
\end{align*}
and  similarly
\begin{align*}
&H_{f^{\ell+km_0}_*\nu_n}\big(\bigvee_{i=1}^{m_0}f^{-i}(\alpha)\big|\bigvee_{i=0}^{km_0+\ell}f^{km_0+\ell-i}(\cP_i^\ell)\big)\\
&=H_{f^{\ell+km_0}_*\nu_n}\big(\bigvee_{i=1}^{m_0-1}f^{-i}(\alpha)\big|\bigvee_{i=0}^{km_0+\ell}f^{km_0+\ell-i}(\cP_i^\ell)\big)+H_{f^{\ell+(k+1)m_0}_*\nu_n}\big(\alpha \big|\bigvee_{i=0}^{(k+1)m_0+\ell-1}f^{(k+1)m_0+\ell-i}(\cP_i^\ell)\big).
\end{align*}
For notational convenience, let us denote
$$g_{k,\ell}:=f^{(k+1)m_0+\ell} \quad \text{ and }\quad 
\cP^\ell(k)=\bigvee_{i=0}^{(k+1)m_0+\ell-1}f^{(k+1)m_0+\ell-i}(\cP_i^\ell).$$
In order to prove the lemma, we have to compare $H_{(g_{k,\ell})_*\nu_n}\big(\alpha |\cP^\ell(k)\big)$ with $H_{(g_{k,\ell})_*\nu_n}\big(\beta |\cP^\ell(k)\big)$ for each $\ell\in\{0,m_0-1\}$ and $k\in\{0,[\frac{n-\ell}{m_0}]-1\}$.

For each $B\in\cP^\ell(k)$, let $\alpha|_B$ and $\beta|_B$ be the partitions
on $B$ induced by $\alpha$ and $\beta$ respectively, and $\cP^\ell_{\neq}(k)$ be the set of $B\in\cP^\ell(k)$ such that $\alpha|_B\neq \beta|_B$. Then since $\alpha\prec\beta$,
\begin{align*}
&\big|H_{(g_{k,\ell})_*\nu_n}\big(\alpha |\cP^\ell(k)\big)-H_{(g_{k,\ell})_*\nu_n}\big(\beta |\cP^\ell(k)\big)\big|\\
&\hspace{1cm}=\big|\sum_{B\in\cP^\ell_{\neq}(k)}(g_{k,\ell})_*\nu_n(B)\big( H_{(g_{k,\ell})_*\nu_n|_{B}}(\alpha|_B)-H_{(g_{k,\ell})_*\nu_n|_{B}}(\beta|_B)\big)\big|
\leq \log\#\beta\cdot \sum_{B\in\cP^\ell_{\neq}(k)}(g_{k,\ell})_*\nu_n(B).
\end{align*}

We now localize the support of $(g_{k,\ell})_*\nu_n$:

\begin{Claim}
	For each $B\in\cP^\ell(k)$, the measure $(g_{k,\ell})_*\nu_n\big|_{B}$ is supported on a disc $D_B$ tangent to the cone field $Dg_{k,\ell}(\cC^{u})$. Moreover, one has $\diam(f^{-i}(D_B))<\rho$ for $i=1,\cdots, (k+1)m_0+\ell.$ 
\end{Claim}
\proof By the choice of $B$ and since $\diam(\beta)\leq\diam(\alpha)<\rho$ (first item of Theorem~\ref{thm.switch-to-finite}),
$$\diam(f^{-i}(B))<\rho \textrm{\:for $i=1,\cdots, (k+1)m_0+\ell.$}$$ 
Since $\nu_n$ is supported on $D$ which is tangent to $\cC^{u}$, the image $(g_{k,\ell})_*\nu_n\big|_{B}$ is supported on the union of finitely many disjoint discs in $g_{k,\ell}(D)$ of diameter $\rho$ and  tangent to the cone field $Dg_{k,\ell}(\cC^{u})$. All backward iterates by $f^{-i}$, for $i\in[1, (k+1)m_0+\ell ]$,
remain $\rho$-close and tangent to $\cC^u_\e$; moreover $\nu_n$ is supported on a single disc $D$.
Hence $(g_{k,\ell})_*\nu_n\big|_{B}$ can only be contained in a single disc.
\endproof

For $(k+1)m_0+\ell\geq n_\e$ and $B\in\cP^\ell_{\neq}(k)$,
the third item of Theorem~\ref{thm.switch-to-finite} for $D_B$ and $V_\e$ gives
$$\supp((g_{k,\ell})_*\nu_n |_B)\subset M\setminus V_\e.$$
Now, the left hand side in~\eqref{e.alpha-beta} is bounded by
\begin{align*}
&\sum_{\ell=0}^{m_0-1}\sum_{k=0}^{[\frac{n-\ell}{m_0}]-1}
\log\#\beta\cdot\sum_{B\in\cP_{\neq}^\ell(k)}(g_{k,\ell})_*\nu_n(B)
\leq \log\#\beta\cdot\bigg(n_\e+\sum_{\ell=0}^{m_0-1}\sum_{k=0}^{[\frac{n-\ell}{m_0}]-1} (g_{k,\ell})_*\nu_n(M\setminus V_\e)\bigg)\\
&\hspace{2cm} \leq n\cdot \log\#\beta\cdot\bigg(\mu_n(M\setminus V_\e)+\frac{n_\e}{n}\bigg)
\;\;\leq\;\; n\cdot \log\#\beta\cdot\bigg(\frac{\delta}{2}+\frac{n_\e}{n}\bigg).
\end{align*}
By our choice of $\delta$,
this is smaller than $\frac{n\e}4$ provided $n$ is larger or equal to any
$N_2>2n_\e/\delta$.
\endproof
 
\subsection{Proof of Theorem~\ref{thm.volume-control-by-entropy}}
Let $\eta=\min\{\eta_0,\eta_1,\eta_2\}$, where $\eta_0,\eta_1,\eta_2$ are given in Theorem~\ref{thm.switch-to-finite}
(applied for $\varepsilon/4$), Proposition~\ref{p.volume-through-separating-set} and lemma~\ref{l.error-estimate} respectively.
We also get $c_\e$ and $ N_2$ which do not depend on $D$.

Recall~\eqref{e.estimateH}.
Applying successively Lemma~\ref{l.error-estimate}, the concavity of the entropy with respect to the measure,
and the second item of Theorem~\ref{thm.switch-to-finite} (since $d(\mu_n,\mu)<\eta$), we get for $n\geq N_2$
\begin{align*}
\sum_{\ell=0}^{m_0-1}H_{\nu_n}\big(\bigvee_{i=0}^{n-1}f^{-i}(\cP^\ell_i)\big)
&\leq 2m_0^2\cdot\log\#\beta+\frac{n\e}{4}
+\sum_{\ell=0}^{m_0-1}\sum_{k=0}^{[\frac{n-\ell}{m_0}]-1}H_{f^{\ell+km_0}_*\nu_n}\big(\bigvee_{i=1}^{m_0}f^{-i}(\alpha)|\beta\big)\\
&\leq 2m_0^2\cdot\log\#\beta+\frac{n\e}{4}+n\cdot H_{\mu_n}\big(\bigvee_{i=1}^{m_0}f^{-i}(\alpha)|\beta\big)\\
&\leq 2m_0^2\cdot\log\#\beta+\frac{nm_0\e}{2} +n\cdot m_0\cdot h_{\mu}(f,\cF^u).\end{align*}
Proposition~\ref{p.volume-through-separating-set} gives
$$\Leb_D\big(C_{n}(\mu,\eta)\cap D\cap\bigcap_{i=0}^{n-1}f^{-i}(U)\big)\leq
	c_\e\cdot \exp\big({\textstyle{2m_0\cdot\log\#\beta+n\e+n\int\psi\ud\mu}}+n\cdot h_{\mu}(f,\cF^u)\big).$$

Choosing $c\gg c_\e\cdot \exp{\big(2m_0\cdot\log\#\beta\big)}$
gives the estimate of Theorem~\ref{thm.volume-control-by-entropy} for any integer $n$.
\qed

\section{Existence of Gibbs u-states: Proofs of Theorems~\ref{thm-III}, \ref{thm-IV}   and    Corollaries ~\ref{coro-II} and ~\ref{coro-III}}~\label{s.Proofs}
We derive some consequences of Theorem~\ref{thm.volume-control-by-entropy}. 
\subsection{Proof of Theorem~\ref{thm-III}}
We prove a more precise result.
\newcounter{theorembis}
\setcounter{theorembis}{\value{theoremalph}}
\setcounter{theoremalph}{\value{theorembisbis}}
\renewcommand{\thetheoremalph}{\Alph{theoremalph}'}
\begin{theoremalph}\label{thm-III-bis}
Consider a $C^1$-diffeomorphism $f$, an $u$-laminated set $\Lambda$ with a partially hyperbolic splitting $E^{cs}\oplus E^{uu}$
and an unstable cone field $\mathcal{C}^u$.
Then there exists a small neighborhood $U$ of $\La$ such that for any disc $D\subset U$ tangent to $\cC^u$,
and for Lebesgue almost every point $x\in D\cap \bigcap_{n\geq 0}f^{-n}(U)$, any limit $\mu$
of the sequence $\frac 1 n \sum_{i=0}^{n-1} \delta_{f^i(x)}$ satisfies
\begin{equation}\label{e.u-gibbs2}
h_\mu(f,\cF^{u})=\int\log{|\det(D f|_{E^{uu}})|}\ud\mu.
\end{equation}
\end{theoremalph}
\setcounter{theoremalph}{\value{theorembis}}
\renewcommand{\thetheoremalph}{\Alph{theoremalph}}
\begin{proof}
The set $\Lambda$ is $u$-laminated.
Let $U$, $\widehat \cC^{u}$ and $r_0>0$ be the open neighborhood of $\La$, the cone field defined in $U$ and the positive number given by Theorem~\ref{thm.volume-control-by-entropy} respectively.
Without loss of generality, one can assume that the disc $D\subset U$ tangent to $\cC^{u}$ has its diameter no more than $r_0$,
and by Remark~\ref{r.cone}, that it is tangent to $\widehat \cC^u$.
By Theorem~\ref{thm.relative-pseudo-physical} and Remark~\ref{r.localize}, for Lebesgue $a.e.$ $x\in D$, any $\mu\in\cM(x)$ is pseudo-physical  relative  to
$Z:=D\cap \bigcap_{n\geq 0}f^{-n}(U)$. 

Let us assume by contradiction that~\eqref{e.u-gibbs2} does not hold.
From the inequality in Theorem~\ref{thm.ruelle}, there exists $\varepsilon>0$ such that 
$\int\log|\det(Df|_{E^{uu}})|\ud\mu-h_\mu(f,\cF^{u})>2\e$. Let $\eta>0$ and $c>0$ be the numbers given after $\mu,\e$ by Theorem~\ref{thm.volume-control-by-entropy}. Note that
  $$\big\{x\in M: \ud(\cM(x),\mu)<\eta \big\}=\bigcup_{l=1}^\infty\bigcap_{k=1}^\infty\bigcup_{n=k}^\infty C_n\big(\mu,{\textstyle\frac{l}{l+1}}\eta\big) \subset \bigcap_{k=1}^\infty\bigcup_{n=k}^\infty C_n(\mu,\eta).$$
Since $\mu$ is pseudo-physical relative to $Z$, for  $\eta>0$, there is $\delta_0>0$ such that for any $k\in \mathbb{N}$,
  $$\Leb_D(Z\cap \cup_{n=k}^\infty C_n(\mu,\eta))>\delta_0.$$
By Theorem~\ref{thm.volume-control-by-entropy}, we have 
 \begin{align*}
\Leb_D(Z\cap \cup_{n=k}^\infty C_n(\mu,\eta))\leq \sum_{n=k}^\infty  c\cdot e^{n(-\int\log|\det(Df|_{E^{uu}})|\ud\mu+h_\mu(f,\cF^{u})+\e)}<c\sum_{n=k}^\infty e^{-n\e},
 \end{align*} 
which contradicts to the fact that $\Leb_D(Z\cap \cup_{n=k}^\infty C_n(\mu,\eta))>\delta_0$ for any $k\in\mathbb{N}$.
\end{proof}

\subsection{Proof of Corollary~\ref{coro-II}}
Let us consider the compact and convex set introduced in Corollary~\ref{c.convex-compact}
$$\mathcal{M}_{\rm u}=\big\{\mu\in\mathcal{M}_{\rm inv}(\La,f): h_\mu(f,\cF^u)=\int\log|\det(Df|_{E^{uu}})|\ud\mu\big\}.$$
Let $U$ and $\cC^{u}$ be the neighborhood of $\La$ and the unstable cone field given by Theorem~\ref{thm.volume-control-by-entropy}.
Recall that $\mathcal{M}(x)$ denotes the accumulation set of the empirical measures of $x$.
For any disc $D$ tangent to $\cC^{u}$, there is a full Lebesgue  measure subset $\tilde{D}\subset D$ such that $\mathcal{M}(x)\subset\mathcal{M}_{\rm u}$ with $x\in\tilde{D}$. 

Let $\{n_i\}$ be an increasing sequence of integers
and let $\mu\in \cM_{\rm inv}(\Lambda,f)$ such that $$\lim_{i\rightarrow\infty}\frac{1}{n_i}\sum_{j=0}^{n_i-1}f^j_*\Leb_{D}^*=\mu,$$
where $\Leb_{D}^*$ denotes the normalized Lebesgue measure on $D$.

For any $\e>0$, consider the $\e$-neighborhood $\mathcal{V}_\e$ of $\mathcal{M}_{\rm u}$ and the set
$$D_k=\big\{x\in D:\frac{1}{n_i}\sum_{j=0}^{n_i-1}f^j_*\delta_x\in\mathcal{V}_\e, \textrm{\:\: for $i\geq k$}\big\}.$$
Theorem~\ref{thm.volume-control-by-entropy} gives $\lim_{k\rightarrow\infty}\Leb_D^*(D_k)=1$.
Take $k_0$ so that $\Leb_D^*(D_{k_0})\geq 1-\e$.
For $i\geq k_0$,
$$\frac{1}{n_i}\sum_{j=0}^{n_i-1}f^j_*\Leb_{D}^*=\int_{D_{k_0}}\frac{1}{n_i}\sum_{j=0}^{n_i-1}f^j_*\delta_x~\ud\Leb_D^*+
\int_{D\setminus D_{k_0}}\frac{1}{n_i}\sum_{j=0}^{n_i-1}f^j_*\delta_x~\ud\Leb_D^*.$$
The choice of $D_{k_0}$ and the convexity of $\mathcal{V}_\e$ immediately give:
\begin{Claim}
For each $i\geq k_0$, there exists an invariant measure $\nu_i\in\mathcal{M}_{\rm u}$ such that
$$\ud\bigg({\textstyle \frac{1}{\Leb^*_D(D_{k_0})}}\int_{D_{k_0}}\frac{1}{n_i}\sum_{j=0}^{n_i-1}f^j_*\delta_x~\ud\Leb_D^*\;,\;\nu_i\bigg)<\e.$$
\end{Claim}
For any continuous function $\varphi:M\mapsto\mathbb{R}$, one has
\begin{eqnarray*}
&&\bigg|\int\varphi~\ud ({\textstyle\frac{1}{n_i}\sum_{j=0}^{n_i-1}f^j_*\Leb_{D}^*})\;-\;\int\varphi~\ud\nu_i\bigg|\\
&&=\bigg|\int_{D_{k_0}}\int\varphi~\ud({\textstyle\frac{1}{n_i}\sum_{j=0}^{n_i-1}f^j_*\delta_{x}})\ud\Leb_{D}^*\;-\;\int\varphi~\ud\nu_i+\int_{D\setminus D_{k_0}}\int\varphi~\ud({\textstyle\frac{1}{n_i}\sum_{j=0}^{n_i-1}f^j_*\Leb_{D}^*})~\ud\Leb_D^*\bigg|\\
&&\leq \bigg|\int\varphi\ud({\textstyle\frac{1}{\Leb^*_D(D_{k_0})}\int_{D_{k_0}}\frac{1}{n_i}\sum_{j=0}^{n_i-1}f^j_*\delta_x})~\ud\Leb_D^*\;-\;\int\varphi\ud\nu_i\bigg|
+\big({\textstyle\frac{1}{\Leb^*_D(D_{k_0})}}+\e-1\big)\cdot\|\varphi\|
\end{eqnarray*}
which implies
$$\ud\bigg(\frac{1}{n_i}\sum_{j=0}^{n_i-1}f^j_*\Leb_{D}^*\;,\;\nu_i\bigg)<\frac{1}{1-\e}-1+2\e.$$
Then $\ud(\mu,\mathcal{M}_{\rm u})\leq \frac{1}{1-\e}-1+2\e$, hence $\mu\in\mathcal{M}_{\rm u}$ since $\mathcal{M}_{\rm u}$ is compact.
\qed

\subsection{Existence of SRB measures: Proof of Corollary~\ref{coro-III}}
We prove the following stronger result:
\begin{Corollary}\label{c.SRB}
Consider a $C^1$ diffeomorphism $f$ and an attracting set $\La$ with a partially hyperbolic splitting
$T_\La M=E^{ss}\oplus E^c\oplus E^{uu}$ such that $\dim(E^c)=1$.
Then for Lebesgue almost every point $x$ in a neighborhood of $\Lambda$,
and for each $\mu\in\cM(x)$,
	\begin{itemize}
	\item either the center Lyapunov exponent of each ergodic component of $\mu$ is non-negative center Lyapunov exponent; in particular, $\mu$ is an SRB measure;
	\item or there exist ergodic components of $\mu$ that are SRB measures with negative center Lyapunov exponent.
	\end{itemize}
\end{Corollary} 
\noindent
\begin{proof}
From equations~\eqref{e.pesin} and~\eqref{e.u-gibbs}
given by Theorems~\ref{thm-III} and~\ref{thm.volume-bound} in Appendix~\ref{Appendix-A},
for Lebesgue almost every point $x$
in the attracting basin of $\Lambda$,
each limit measure $\mu\in\cM(x)$ satisfies
$$h_{\mu}(f,\cF^u)=\int\log|\det(Df|_{E^{uu}})|\ud\mu \textrm{ and }h_\mu(f)\geq \int\log|\det(Df|_{E^c\oplus E^{uu}})|\ud\mu.$$
Corollary~\ref{c.convex-compact} gives $h_\nu(f,\cF^u)=\int\log|\det(Df|_{E^{uu}})|\ud\nu$ for each ergodic component $\nu$ of $\mu$.

If each ergodic component $\nu$ of $\mu$ has non-negative center Lyapunov exponent, then
$$\int\sum\lambda^+(z)\ud\nu(z)=\int\log|\det(Df|_{E^c\oplus E^{uu}})|\ud\nu.$$
Combining this with~\eqref{e.pesin} and Ruelle inequality, one gets
$$h_\mu(f)\geq\int\log|\det(Df|_{E^c\oplus E^{uu}})|\ud\mu=\int\sum\lambda^+(z)\ud\mu(z)\geq h_\mu(f);$$
therefore $\mu$ is an SRB measure.

If there are ergodic components $\nu$ with negative  center Lyapunov exponent, they satisfy
$$\int\log|\det(Df|_{E^{uu}})|\ud\nu =\int\sum\lambda^+(z)\ud\nu(z).$$
The equation~\eqref{e.u-gibbs} for $\nu$ and Ruelle inequality then give
$$\int\sum\lambda^+(z)\ud\nu(z)=\int\log|\det(Df|_{ E^{uu}})|\ud\nu=h_\nu(f,\cF^u)\leq h_\nu(f)\leq \int\sum\lambda^+(z)\ud\nu(z),$$ therefore, $\nu$ is an SRB measure with negative center Lyapunov exponent.
\end{proof}

\subsection{Large deviation for continuous functions: Proof of Theorem~\ref{thm-IV}}~\label{s.large-deviation}
We prove a stronger version of Theorem~\ref{thm-IV}.

\newcounter{theorembisD}
\setcounter{theorembisD}{\value{theoremalph}}
\setcounter{theoremalph}{\value{theorembisbisD}}
\renewcommand{\thetheoremalph}{\Alph{theoremalph}'}
\begin{theoremalph}\label{thm.large-deviation-for-lamination}
Let $f$ be a $C^1$-diffeomorphism and
$\La$ be a $u$-laminated set with a partially hyperbolic splitting $T_\La M=E^{cs}\oplus E^{uu}$. 
Then
for any continuous function $\varphi: M\to\mathbb{R}$ and any $\e>0$,
there exist a neighborhood $U_\e$ of $\La$, a $Df$-invariant cone field $\cC^{u}$ on $U_\e$ and $r_0, a_\e,b_\e>0$ such that
for any disc $D$ tangent to $\cC^{u}$ of diameter smaller than $r_0$ and any $n\in\mathbb{N}$,
\begin{equation}\label{e.deviation}
\Leb\bigg\{x\in D:\; x\in\bigcap_{i=0}^{n-1}f^{-i}(U_\e)\textrm{ and } \ud\bigg(\frac{1}{n}\sum_{i=0}^{n-1}\varphi(f^i(x)), I(\varphi)\bigg)>\e \bigg\}< a_\e\cdot e^{-n\cdot b_\e}.
\end{equation}
\end{theoremalph}
\setcounter{theoremalph}{\value{theorembisD}}
\renewcommand{\thetheoremalph}{\Alph{theoremalph}}
As before $I(\varphi)$ is the interval defined by
	$$I(\varphi):=\bigg\{\int\varphi\ud\mu:\;\; \textrm{$\mu\in\mathcal{M}_{\rm inv}(\La,f)$ satisfies $h_\mu(f,\cF^u)=\int\log|\det(Df|_{E^{uu}})|\ud\mu$} \bigg\}.$$
This immediately implies Theorem~\ref{thm-IV}: when $\Lambda$ is an attracting set,  any unstable disc
in a neighborhood of $\Lambda$ will eventually be contained in
$\bigcap_{i=0}^{\infty}f^{-i}(U_\e)$ one then apply
Theorem~\ref{thm.large-deviation-for-lamination} to the unstable leaves
of foliated domains covering $\Lambda$ and Fubini Theorem.

\begin{proof}[Proof of Theorem~\ref{thm.large-deviation-for-lamination}]
Let $U$, $\cC^{u}$ and $r_0$ be the neighborhood of $\La$, the $Df$-invariant  cone field defined in $U$ and the positive number given by Theorem~\ref{thm.volume-control-by-entropy}. 
For $\varphi\colon M\to \RR$ and $\e>0$,
let $I_\e\subset \RR$ be the $\e/2$-open neighborhood of $I(\varphi)$ and
$I_\e^c$ its complement and let us denote
$$\mathcal{N}_\epsilon:=\bigg\{\mu\in\mathcal{M}_{\rm inv}(\La,f): \int\varphi\ud\mu\in I_\e^c\bigg\}.$$
For each $\mu\in\cM_{\rm inv}(\La,f)$, let
$$b_\mu:= {\textstyle \frac 1 2} \big(\int\log|\det(Df|_{E^{uu}})|\ud\mu-h_\mu(f,\cF^u)\big)$$
 and let $\eta_{\mu},c_\mu>0$ be the numbers given after $\mu,b_\mu$ by
Theorem~\ref{thm.volume-control-by-entropy}:
for each compact disc $D$ tangent to $\cC^{u}$ of diameter smaller than $r_0$ and any $n\in\mathbb{N}$, one has 
\begin{equation}\label{e.cover}
\Leb(  D\cap \bigcap_{i=0}^{n-1}f^{-i}(U)\cap C_{n}(\mu,2\eta_\mu))<c_\mu\cdot e^{-n b_\mu}.
\end{equation}
We choose the $\eta_\mu$ small so that
for any probability measure $\nu$ satisfying $\ud(\mu,\nu)<\eta$,
\begin{equation}\label{e.cont-measure}
\bigg|\int\varphi\ud\mu-\int\varphi\ud\nu\bigg|<\frac{\e}{4}.
\end{equation}
Since $\mathcal{M}_{\rm inv}(\La,f)$ is a compact set, there exist  $\mu_1,\cdots,\mu_k\in\cN_\e$ such that 
$$\mathcal{M}_{\rm inv}(\La,f)\subset\cup_{j=1}^k B_{\eta_{\mu_j}}(\mu_j),$$
where $B_{\eta_{\mu_j}}(\mu_j)$ denotes the $\eta_{\mu_j}$-neighborhood of $\mu_j$ in the space of probability measures on $M$.
We denote $\eta_j=\eta_{\mu_j}$ and $c_j=c_{\mu_j}$ for simplicity.
The following lemma gives $U_\e.$
\begin{Lemma}~\label{l.amount-to-iterate}
	There exists a neighborhood $U_\e$ of $\Lambda$
	and an integer $N_\e$ such that for any $n\geq N_\e$ and any $x\in \bigcap_{i=0}^{n-1}f^{-i}(U_\e)$,
	$$\frac{1}{n}\sum_{i=0}^{n-1}\delta_{f^i(x)}\in \bigcup_i B_{\eta_i}(\mu_i).$$	
\end{Lemma}
\begin{proof}
Assume,  on the contrary, that the $1/k$-neighborhood $U(1/k)$ of $\La$, there exists
$n$ arbitrarily large and $x\in \bigcap_{i=0}^{n-1}f^{-i}(U(1/k))$ such that $\frac{1}{n}\sum_{i=0}^{n-1}\delta_{f^i(x)}\not \in \bigcup_i B_{\eta_i}(\mu_i)$.
Then taking the limit as $n\to +\infty$ gives an invariant measure $\nu_k$ on the maximal invariant set of $U(1/k)$
which does not belong to $\bigcup_i B_{\eta_i}(\mu_i)$. Any limit of the $\nu_k$ is an invariant measure on $\Lambda$
that does not belong to $\bigcup_i B_{\eta_i}(\mu_i)$, which contradicts~\eqref{e.cover}.
\end{proof}

We can conclude the proof of Theorem~\ref{thm.large-deviation-for-lamination}.
For any $x\in \bigcap_{i=0}^{n-1}f^{-i}(U_\e)$ and $n\geq N_\e$, there exists $i_0$ such that
$\frac{1}{n}\sum_{i=0}^{n-1}\delta_{f^i(x)}\in B_{\eta_{i_0}}(\mu_{i_0})$, ie $x\in C_n(\mu_{i_0},\eta_{i_0})$.
From~\eqref{e.cont-measure}, if one has
$$\bigg|\frac{1}{n}\sum_{i=0}^{n-1}\varphi(f^i(x))-\int\varphi\ud\mu\bigg|\geq \e,$$
then $\bigg|\int\varphi\ud\mu_{i_0}-I(\varphi)\bigg|\geq\e/2,$ hence $\mu_{i_0}$ belongs to $\cN_\e$.

Let $I\subset \{1,\cdots,k\}$ such that $\{\mu_1,\cdots,\mu_k\}\cap \cN_\e=\{\mu_i\}_{i\in I}$. 
Then we have for $n\geq N_\e$
$$
\Leb\bigg\{x\in D\cap \bigcap_{i=0}^{n-1}f^{-i}(U_\e):  \bigg|\frac{1}{n}\sum_{i=0}^{n-1}\varphi(f^i(x))-\int\varphi\ud\mu\bigg|\geq\e \bigg\}
\;\leq \sum_{i\in I} \Leb(C_n(\mu_i,\eta_i))\;\leq \sum_{i\in I} c_{i} e^{-n b_{\mu_i}} .
$$

Now, we only need to consider an upper bound $a$ of the volume of discs tangent to the cone field
$\cC^u$ and with diameter smaller than $r_0$ and to set
$$b_\e=\min\{b_{\mu_1},\cdots,b_{\mu_k}\}
\text{ and }
a_\e=\max\bigg\{a\cdot e^{N_\e\cdot b_\e},\sum_{i=1}^k c_i\bigg\}.$$
\end{proof}

\section{SRB measures for $C^{1+\alpha}$ partially hyperbolic diffeomorphisms}\label{s.example}
We focus on $C^{1+\alpha}$-diffeomorphisms for $\alpha>0$ and partially hyperbolic sets with one-dimensional center. 
\subsection{Existence of the center Lyapunov exponent: proof of Theorem~\ref{thm-I}}
Before Theorem~\ref{thm-I} we prove two preliminary results.
\begin{Proposition}\label{Prop:srb-positive-lyapunov}
Let $f$ be a $C^{1+\alpha}$-diffeomorphisms, $\alpha>0$,
and let $\mu$ be an ergodic SRB measure whose support admits a partially hyperbolic splitting $E^{ss}\oplus E^{cu}$, and whose Lyapunov exponents along $E^{cu}$ are all positive. Then there is an open invariant set $O(\mu)$ such that
$$\Leb(O(\mu)\triangle{\rm Basin}(\mu))=0
\text{ and } O(\mu)\cap \supp(\mu)\neq\emptyset.$$
In particular, for Lebesgue a.e. point $x\in M$, if $\omega(x)\supset \supp(\mu)$ then $x\in {\rm Basin}(\mu)$.
\end{Proposition}
\begin{Remark}
Kan's example~\cite{kan} shows that
the basin of $\mu$ may not be essentially open as
in Proposition~\ref{Prop:srb-positive-lyapunov}
when the first bundle is not uniformly contracted.
\end{Remark}

\begin{proof}[Proof]
 By~\cite[Theorem 3.4]{L}, the disintegration of $\mu$ along the unstable manifolds is absolutely continuous with respect to the Lebesgue measure on the unstable manifolds. Consequently there exists a disc $D\subset \supp(\mu)$ tangent to $E^{cu}$ such that the basin of $\mu$ contains a set $X\subset D$ with full
 Lebesgue measure.
 The union of the strong stable manifolds of points of $D$ contains a nonempty open set
 $O_0$ which intersects $\supp(\mu)$. The union of the local strong stable leaves of points of $X$
is absolutely continuous~\cite{P}.
 Consequently, the union of the strong stable manifolds of points of $X$ has full Lebesgue measure in $O_0$.
 This proves that ${\rm Basin}(\mu)$ has full Lebesgue measure in the open set
 $O(\mu):=\bigcup_{n\in \mathbb{Z}} f^n(O_0)$. The orbit of every point $x$ in the basin of $\mu$,
 accumulates any point of $\supp(\mu)$, hence enters in $O_0$.
 Up to removing  the invariant set $O(\mu)\setminus {\rm Basin}(\mu)$ (which has zero Lebesgue measure),
 one concludes that the orbit of Lebesgue almost every point in ${\rm Basin}(\mu)$ is contained in $O(\mu)$.
 Hence $O(\mu)$ and ${\rm Basin}(\mu)$ coincide modulo a set with zero Lebesgue measure.
\end{proof}

\begin{Proposition}~\label{p.physical-negative}
Let $f$ be a $C^{1+\alpha}$-diffeomorphism, $\alpha>0$, and
$\mu$ a hyperbolic SRB measure.
Then Lebesgue almost every point $x\in M$ satisfying:
\begin{itemize}
\item[--] $\omega(x)$ has a partially hyperbolic splitting $E^{cs}\oplus E^{uu}$,
\item[--] $\omega(x)$ contains $\supp(\mu)$ and the Lyapunov exponents of $\mu$ along $E^{cs}$ are all negative,
\end{itemize}
belongs to the basin of $\mu$.
\end{Proposition}
\proof
Let us consider the set $\mathcal{L}$ of
partially hyperbolic sets containing $\supp(\mu)$
with a splitting $E^{cs}\oplus E^{uu}$,
such that the Lyapunov exponents of $\mu$ along $E^{cs}$ are all negative.
There exists a countable sequence $(\Lambda_n)$ in $\mathcal{L}$ such that
any set $\Lambda\in\cL$ is contained in one of the $\Lambda_n$.
As a consequence, it is enough to fix a set $\Lambda\in\Lambda$
and to prove the proposition for Lebesgue almost every point $x$
such that $\supp(\mu)\subset \omega(x)\subset \Lambda$.

Let $\cC^u$ be an unstable cone field on a neighborhood $U$ of $\Lambda$
which satisfies the Theorem~\ref{thm.u-density-point} and $\cA:=\big\{x\in U: \textrm{$\omega(x)\supset \supp(\mu)$ and $f^n(x)\in U$ for all $n\geq 0$}\big\}$. Then one has $f(\cA)\subset \cA$. It is enough to show that
$\cN:=\cA\setminus {\rm Basin}(\mu)$ has zero Lebesgue measure.
Assume, on the contrary,  that $\Leb(\cN)>0$. Thus, there exists a disc $D$ tangent to $\cC^{u}$ such that $\Leb_D(\cN\cap D)>0.$

By~\cite[Theorem 3.4]{L}, there exists an unstable disc $\Delta$ in the support of $\mu$
and a set $X\subset \Delta$ with positive Lebesgue measure such that
any point in $X$ is in the basin of $\mu$ and has a stable manifold tangent to $E^{cs}$.
By~\cite{P}, up to reducing $X$, one can assume that the (local) stable manifolds of points $z\in X$
vary continuously with $z$ and induce an absolutely continuous lamination $W^s_{loc}(X)$.
We fix a density point $z_0\in X$ of $X$ inside $\Delta$.

In order to define unstable density basis inside $D$, we fix $\delta>0$ small.
By Theorem~\ref{thm.u-density-point}, the set of unstable density points of $\cN\cap D$ has full Lebesgue
measure in $\cN\cap D$, and we fix $x\in \cN\cap D$ one of them.
There exists a sequence $n_k\to +\infty$ such that $f^{n_k}(x)\to z_0$.
The density basis $B_{D,n}(x)$ satisfy
$$\frac{\Leb(B_{D, n_k}(x)\cap\cN)}{\Leb(B_{D, n_k}(x))}\underset{k\to +\infty}\longrightarrow 1.$$
The definition of $\cN$,  the density basis and the bounded distortion along the unstable manifold
(using the uniform expansion and the $C^{1+\alpha}$-smoothness) imply:
\begin{equation}\label{e.density-measure}
\frac{\Leb(B_{f^{n_k}(D)}(f^{n_k}(x),\delta)\cap\cN)}{\Leb(B_{f^{n_k}(D)}(f^{n_k}(x),\delta))}\underset{k\to +\infty}\longrightarrow 1.
\end{equation}

Since  $f^{n_k}(x)$ converges to $z_0$ and  the unstable cones converge to the unstable bundle
under forward iterations, the disc $B_{f^{n_k}(D)}(f^{n_k}(x),\delta)$ gets arbitrarily close to
$B_\Delta(z_0,\delta)$ for the $C^1$-topology.
The absolute continuity of the stable lamination over $X$ implies that
for $n_k$ large enough, the Lebesgue measure of $W^s_{loc}(X)\cap B_{f^{n_k}(D)}(f^{n_k}(x),\delta)$
is positive and uniformly bounded away from zero.
With~\eqref{e.density-measure} this implies that for $k$ large enough
$\cN$ intersects $W^s_{loc}(X)$. This is a contradiction since $W^s_{loc}(X)\subset {\rm Basin}(\mu)$.
\endproof
\medskip

Now, we are ready for proving the existence of the center Lyapunov exponent.
\proof[\bf Proof of Theorem~\ref{thm-I}]
Let $U$ be an attracting neighborhood of $\La$.
\begin{Lemma}
The bundle $E^c$ admits a (non-unique) continuous and invariant extension to $U$.
Moreover, for any $x\in U$ we have
$$\lim_{n\to+\infty} \frac{\|Df^n|_{E^c}(x)\|}{\|Df^n|_{E^{cs}}(x)\|}=1.$$
\end{Lemma}
\begin{proof}
Let us consider a continuous extension $E\subset E^{cs}$ of $E^c$: up to shrinking the open set $U$,
one can assume that $E$ is defined on $U$ and is contained in a center-unstable cone.
Using a cut-off function, one can interpolate $E$ with $Df(E)$ and get a continuous extension
$E'$ of $E^c$ such that $E'(f(x))=Df(E'(x))$ for any $x\in U$ outside a small neighborhood of
$f(\overline U)$. One then define $E^c$ on $U\setminus f(U)$ as follows:
for $x\in f^n(U)\setminus f^{n+1}(U)$ we set $E^c(x)=Df^n(E'(f^{-n}(x)))$.
By construction $E^c$ is continuous and invariant on $U\setminus \Lambda$.

The dominated slitting $E^{cs}=E^{ss}\oplus E^c$ and the cone field criterion (see~\cite{CP}) implies that $Df^n(E')$ converges to $E^c|_{\Lambda}$.
Hence the extension of $E^c$ is also continuous at points of $\Lambda$.
\end{proof}
The previous lemma shows that the center Lyapunov exponent of any point $x\in U$
can be studied by considering the Birkhoff averages of the continuous function
$$\varphi\colon x\mapsto \log\|Df|_{E^c}(x)\|.$$

Propositions~\ref{Prop:srb-positive-lyapunov},~\ref{p.physical-negative} show that
Lebesgue almost every point $x$ in the set
$$U^{h}:=\bigg\{x\in U: \; \omega(x) \textrm{\: carries a hyperbolic ergodic SRB measure}\bigg\}.$$
belongs to the basin of a hyperbolic ergodic SRB measure $\mu$, which by~\cite[Theorem 4.9]{L} is physical.
Hence the Birkhoff averages of $\varphi$ along the forward orbit of $x$ converge.
The limit
$$\lambda^c(x):=\lim_{n\to+\infty} \frac 1 n \log \|Df^n|_{E^c}(x)\|=\lim_{n\to+\infty} \frac 1 n \log \|Df^n|_{E^{cs}}(x)\|$$
exists and coincides with the center Lyapunov exponent
$\int \log\|Df|_{E^c}(x)\| \ud\mu(x)=\int\varphi \ud\mu$.
In particular $\lambda^c(x)$ does not vanish.

Then for Lebesgue a.e $x\in U\setminus U^{h}$, the $\omega$-limit set of $x$ does not carry any hyperbolic SRB measures. By Corollary~\ref{c.SRB} this implies that each limit measure $\mu\in\cM(x)$ is SRB and has a vanishing
center exponent $\int\varphi \ud\mu$.
Since $\varphi$ is continuous this shows that
$$\frac 1 n \sum_{k=0}^{n-1} \varphi(f^k(x))\;=\;  \frac 1 n \log \|Df^n|_{E^c}(x)\|\underset{n\to +\infty}\longrightarrow 0.$$
Hence the center Lyapunov $\lambda^c(x)$ of $x$ is also well-defined in this case
and vanishes.
\endproof

\subsection{An example exhibiting historical behavior: proof of Theorem~\ref{thm-II}}
The example described in Theorem~\ref{thm-II} is obtained by compactification
of a skew translation over an Anosov system. It is well-known that the dynamics of these infinite systems
share properties with the \emph{Brownian motion} on $\mathbb{R}$: this will allow us to study precisely the
asymptotic of the empirical measures.

\subsubsection{Limit properties of skew translations}
We first state classical properties of skew translations.

\begin{Proposition}\label{p.skew}
Let $A$ be a smooth Anosov diffeomorphism on $\TT^2$ preserving a smooth volume $m$ and having at least two fixed points $p,q$.
Let $\phi:\TT^2\to\mathbb{R}$ be a smooth function with $\int\phi\ud m=0$ such that $\phi(p)$, $\phi(q)$ are rationally independent.
Then:
\begin{itemize}

\item[(i)] The measure $m\times \mathrm{Leb}$ is ergodic for the diffeomorphism
$g$ of $\mathbb{T}^2\times \mathbb{R}$ defined by
\begin{equation}\label{e.def-g}
g(x,t)=(A(x),t+\phi(x)).
\end{equation}

\item[(ii)] The number $\sigma:=\sum_{n\in\ZZ}\int \phi\cdot\phi\circ A^n\ud m$ is well-defined and positive.

\item[(iii)] For $m$-almost every point $x\in \mathbb{T}^2$,
the continuous functions $X_n\in C([0,1])$ defined by
$$X_n(t):=\frac{1}{\sqrt{\sigma\cdot n}}\int_0^{nt}\phi(A^{[s]}(x))\ud s,$$
induce a random process which converges weakly to the standard Wiener measure.
\end{itemize}
\end{Proposition}
\begin{proof}
Since $\phi(p)$ and $\phi(q)$ are rationally independent, there do not exist
$\lambda\in \RR$ and $\psi\colon M\to \mathbb{R}$ such that
$\phi=\psi\circ A-\psi \mod[\lambda]$. The ergodicity (i) follows from~\cite[Corollary 3]{guivarch}.

The convergence of the sum defining $\sigma$ is a consequence of the exponential decay of the correlations,
see for instance~\cite[Theorem 3.9]{Liverani}.
Note that $\sigma$ is non-negative, because of 
$$\sum_{n\in\ZZ}\int \phi\cdot\phi\circ A^n\ud m=\lim_{n\rightarrow+\infty}\frac{1}{n}\int\big(\sum_{i=0}^{n-1}\phi\circ A^i\big)^2\ud m.$$
Since $\phi(p)\neq 0$, there is no continuous solution $\psi\colon M\to \RR$
to the cohomological equation
$$\phi=\psi\circ A-\psi.$$
Then in restriction to any $A$-invariant set with full measure for $m$,
there is no measurable solution, see~\cite[Theorem 9]{Li}.
One deduces that $\sigma$ does not vanish (see~\cite[Proposition 4.12]{PaPo}). This gives the second item.
The third item is now~\cite[Corollary 4]{D} for conservative Anosov diffeomorphisms
(see also~\cite[Corollary 3]{DP}).
\end{proof}

\subsubsection{Compactification of the skew translation}
We denote $\mathbb{T}=\mathbb{R}/\mathbb{Z}$.
Any skew translation over an Anosov diffeomorphism on $\mathbb{T}^2$ can be embedded as a partially hyperbolic diffeomorphism
on $\mathbb{T}^3$.

\begin{Proposition}\label{p.compactification}
Let us consider a smooth Anosov diffeomorphism $A$ on  $\mathbb{T}^2$,
a smooth function $\phi:\TT^2\to\mathbb{R}$ and the diffeomorphism $g$ on $\mathbb{T}^2\times \mathbb{R}$ defined by~\eqref{e.def-g}.
Then there exists a smooth diffeomorphism $f$ on $\mathbb{T}^3$
preserving a partially hyperbolic splitting $E^{ss}\oplus E^c\oplus E^{uu}$ such that:
\begin{itemize}
\item[--] the foliation by circles $\{x\}\times \mathbb{T}$ is preserved and tangent to $E^c$;
\item[--] $f$ preserves each torus $\mathbb{T}^2\times \{0\}$ and $\mathbb{T}^2\times \{1/2\}$,
and exchanges $\mathbb{T}^2\times (0,1/2)$ and $\mathbb{T}^2\times (1/2,1)$;
\item[--] the restriction of $f^2$ to $\mathbb{T}^2\times (0,1/2)$
is smoothly conjugated to $g^2$.
\end{itemize}
\end{Proposition}
\begin{proof}
Let $X$ be a smooth vector field on $\mathbb{R}$ such that 
\begin{itemize}
\item $X(t)>0$ for $t\in(0,1/2)$ and $X(0)=X(1/2)=0$,
\item $X$ is $1$-periodic and satisfies $X(-t)=-X(t)$ for each $t\in \RR$.
\end{itemize}
Let $(\Phi_s)_{s\in \RR}$ be the flow induced by $X$ on $\RR$.
The diffeomorphism of $\mathbb{T}^2\times \RR$ defined by
$$F(x,t):=(Ax,-\Phi_{\phi(x)}(t))$$
satisfies $F(x,t+1)=F(x,t)-(0,1)$, hence induces a smooth diffeomorphism $f$ on $\mathbb{T}^3$.
Choosing $X$ arbitrarily close to $0$, the diffeomorphism $f$ is $C^1$-close to
the diffeomorphism $A\times \Id$, hence is partially hyperbolic.
The  first two items then follow.

Note that $f$ commutes with the involution $(x,t)\mapsto (x,-t)$ hence $f^2$ coincides with the diffeomorphism
induced by $$(x,t)\mapsto(A^2x,\Phi_{\phi(A(x))+\phi(x)}(t)).$$
The map $h\colon \mathbb{T}^2\times \RR\to \mathbb{T}^2\times (0,1/2)$ defined by
$(x,s)\mapsto (x,\Phi_s(1/4))$ conjugates the restriction of $f^2$ to $\mathbb{T}^2\times (0,1/2)$
with $g^2$ as claimed in the third item.
\end{proof}

\subsubsection{Historical behavior}
The proof of Theorem~\ref{thm-II} can be concluded as follows.

\begin{Proposition}
Let us consider a smooth Anosov diffeomorphism $A$ of $\mathbb{T}^2$
and a smooth function $\phi\colon \mathbb{T}^2\to \mathbb{R}$ as in Proposition~\ref{p.skew}.
Then the diffeomorphism $f$ of $\mathbb{T}^3$ induced by $A$ and $\phi$ as in Proposition~\ref{p.compactification}
has exactly only two ergodic Gibbs u-states $\nu_1,\nu_2$. Moreover, for Lebesgue almost every $z\in \mathbb{T}^3$,
\begin{itemize}
\item[--] the set of limit measures $\mathcal{M}(x)$ of $x$ is the segment $[\nu_1,\nu_2]$,
\item[--] the orbit of $z$ is dense in $\mathbb{T}^3$.
\end{itemize}
\end{Proposition}
\begin{proof}
Let us recall that $A$ preserves a smooth volume $m$.
By absolute continuity of the stable foliation of $A$, it is the unique Gibbs u-state for $A$
and it is ergodic.
The two measures $\nu_1=m\times \delta_0$ and $\nu_2=m\times \delta_{1/2}$
are $f$-invariant and are Gibbs u-states.

Let us  denote  $S_n\phi(x)=\sum_{j=0}^{n-1}\phi(A^j(x))$ for $x\in \mathbb{T}^2$ and $n\in\NN$.
Then  the skew translation $g$ defined by~\eqref{e.def-g} satisfies
$$g^n(x,t)=(A^n(x), t+S_n\phi(x)).$$ 
We introduce $\cG_n(x):=\{0\leq j\leq n-1: S_j\phi(x)\geq \sqrt{\sigma\cdot n}\}$.

\begin{Claim}
For Lebesgue a.e. $x\in \mathbb{T}^2$
and any $\rho\in(0,1)$, there exists $n$ arbitrarily large such that $$\#\cG_n(x)\geq (1-\rho)\cdot n.$$
\end{Claim} 
\begin{proof}
Let $\cW=\{h\in C^0([0,1], \RR): h(0)=0\}$ endowed with $C^0$-norm.
We consider a continuous function $h:[0,1]\mapsto[0,+\infty)$ such that 
$$h(0)=0 \textrm{ and } h(t)>1 \textrm{ for $t\in[\rho,1]$}.$$
Let $0<\e<\inf_{t\in[\rho,1]}\frac{h(t)-1}{2}$ be small.
Since the Wiener measure has full support in $\cW$,
and since the process $(X_n)$ in $\cW$ defined in Proposition~\ref{p.skew} converges to the Wiener measure
for Lebesgue almost every $x\in \mathbb{T}^2$,
there exists $n$ arbitrarily large such that

$$\sup_{t\in[0,1]} \bigg|\frac{1}{\sqrt{\sigma\cdot n}}\int_0^{nt}\phi(A^{[s]}(x))\ud s-h(t)\bigg|<\e.$$
In particular for any integer $j\in \{0,\dots,n\}$, one has
$$\bigg|\frac{1}{\sqrt{\sigma\cdot n}}S_j\phi(x)-h(j/n)\bigg|<\e.$$
By the definition of $h$ and $\e,$  this gives $S_j\phi(x)>{\sqrt{\sigma\cdot n}}$ for all $j\geq (1-\rho)\cdot n$.
\end{proof}

\begin{Claim}
For Lebesgue a.e. $x\in \mathbb{T}^3$ and all $t\in (0,1/2)$
the measure $\nu_2$ belongs to $\cM(z)$.
\end{Claim}
\begin{proof}
Let $\Gamma$ be a continuous function on $\TT^2\times[0,1]$
and let us fix $\rho>0$ small. Let us consider the set $\cG_n(x)$ for an integer $n$ large given by the previous claim.
One has the estimate 
\begin{align*}
&\big|\frac{1}{n}\sum_{i=0}^{n-1}\Gamma(f^i(x,t))-\frac{1}{n}\sum_{i=0}^{n-1}\Gamma(f^i(x,1/2))\big|\\
&\leq  \frac{1}{n}\sum_{i\in\cG_n(x)}|\Gamma(f^i(x,t))-\Gamma(f^i(x,1/2))|+\frac{1}{n}\sum_{i\notin\cG_n(x)}|\Gamma(f^i(x,t))-\Gamma(f^i(x,1/2))|\\
&\leq \frac{1}{n}\sum_{i\in\cG_n(x)}|\Gamma(A^i(x),\varepsilon_i\cdot \Phi_{S_i\phi(x)}^X(t))-\Gamma(A^i(x),1/2)|+2\rho\cdot \sup  |\Gamma|,
\end{align*}
where $\varepsilon_i=+1$ when $i$ is even and $-1$ when $i$ is odd.

Notice that for $t\in(0,1/2)$, $\varphi_{s}^X(t)$ tends to $1/2$ when $s$ tends to $+\infty.$ By the arbitrariness of $\rho$ and the uniform continuity of $h$, one deduces that the empirical measures $m_{(x,t),n}$ and $m_{(x,1/2),n}$ are close.
\end{proof}
The claim shows that $\nu_2\in\cM(z)$ for Lebesgue a.e. $z\in\mathbb{T}^3$.
Analogously, $\nu_1\in\cM(z)$.

\begin{Claim}
$\nu_1$, $\nu_2$ are the unique ergodic Gibbs u-states.
\end{Claim}
\begin{proof}
Let $\nu$ be an ergodic Gibbs u-state. There is a strong  unstable disc $D$ such that for Leb$_D$ almost every $(x,t)\in D$,
$$\lim_{n\rightarrow+\infty}\frac{1}{n}\sum_{i=0}^{n-1}\delta_{f^i(x,t)}=\nu.$$
The disc $D$ projects to an unstable arc $D'\subset \mathbb{T}^2$ and for Lebesgue almost every $x\in D'$,
the empirical measures converge to the projection of $\nu$. This shows that the projection of $\nu$
to $\mathbb{T}^2$ coincides with $m$ (the unique Gibbs u-state for $A$).

Let us assume by contradiction that $\nu$ is not supported on
$\mathbb{T}^2\times \{0,1/2\}$.
In particular for $\nu$-almost every point $z$, the projection on $\mathbb{T}^2$
belongs to the full $m$-measure set given by the previous Claim.
This implies that the set of limit measures $\cM(z)$ of $z$ contains both $\nu_1$ and $\nu_2$.
This is a contradiction since the empirical measures of $z$ converge to $\nu$ (by Birkhoff ergodic theorem).
\end{proof}

It remains to prove the last statement of the proposition.
From Proposition~\ref{p.skew}, the skew translation $g^2$ is ergodic,
hence from the last item of Proposition~\ref{p.compactification}, the orbit of Lebesgue almost every point $z\in \mathbb{T}^2\times (0,1/2)$
under $f^2$ is dense in $\mathbb{T}^2\times (0,1/2)$.
Since $f$ exchanges the regions $\mathbb{T}^2\times (0,1/2)$ and $\mathbb{T}^2\times (1/2,1)$,
one deduces that the orbit of Lebesgue almost every point $z\in \mathbb{T}^3$ is dense.
\end{proof}

\appendix
 
\section{Generalized Pesin's inequality under a dominated splitting}~\label{Appendix-A}
We sketch here the proof of the inequality~\eqref{e.pesin}  stated in the introduction.
We recall that 
a splitting $T_\Lambda M=E\oplus F$ over  a compact invariant set $\Lambda$ of a diffeomorphism $f$ is \emph{dominated},
if there exists $N\in\mathbb{N}$ such that  
$$\|Df^N|_{E(x)}\|\cdot \|Df^{-N}|_{F(f^N(x))}\|\leq \frac{1}{2}. $$

\begin{theoremalph}[Entropy inequality]\label{thm.volume-bound}
For any $C^1$ diffeomorphism $f$, for any compact invariant  set $\Lambda$ admitting a dominated splitting $E\oplus F$, and for Lebesgue almost every point $x\in M$,
if  $\omega(x)\subset \Lambda$, then each limit measure $\mu\in\cM(x)$ satisfies
\begin{equation}\label{e.pesin2}
h_\mu(f)\geq \int \log|\det Df|_F|d\mu.
\end{equation}
\end{theoremalph}

This improves  a little bit~\cite[Theorem 1]{CCE} and ~\cite[Theorem 4.1]{CaY}.
(We do not assume the semi-continuity of the entropy, nor the existence of a global dominated splitting.)

\begin{Corollary}\label{Coro:improve-CY}
Let $f\in\diff^1(M)$ and $\Lambda$ be an attracting set with the dominated splitting $E\oplus F$.
Then for Lebesgue a.e. $x$ in the basin of $\Lambda$, each limit  measure $\mu\in\mathcal{M}(x)$ satisfies~\eqref{e.pesin2}.
\end{Corollary}

Considering the trivial splitting of $TM$, one gets:
\begin{Corollary}
Let $f\in\diff^1(M)$ with a fixed point $p$. If
$\delta_p$ is physical, then $|\det (Df(p))|\leq  1$.
\end{Corollary}

We also obtain a large deviation result.
\begin{theoremalph}[Large deviation]\label{thm.deviation-domination}
For any $C^1$ diffeomorphism $f$, for any invariant compact set $\Lambda$ admitting a dominated splitting $E\oplus F$,
for any continuous function $\varphi\colon M\to \RR$ and for any $\varepsilon>0$,
there exist a neighborhood $U$ of $\La$ and $a_\e,b_\e>0$ such that 
	$$\Leb\bigg\{ x\in \bigcap_{i=0}^{n-1}f^{-i}(U): \ud\bigg(\frac{1}{n}\sum_{i=0}^{n-1}\varphi(f^i(x)), I(\varphi)\bigg)\geq\e\bigg\}<a_\e\cdot e^{-n b_\e} \quad\textrm{ for any  $n\in\mathbb{N}$},$$
	$$\text{where}\quad I(\varphi):=\bigg\{\int\varphi\ud\mu:\;\; \textrm{$\mu\in\mathcal{M}_{\rm inv}(\La,f)$ satisfies~\eqref{e.pesin2}} \bigg\}.$$
\end{theoremalph}

The main step in the proofs of Theorems~\ref{thm.volume-bound} and~\ref{thm.deviation-domination}
is to bound the measure of the convergent set of invariant measures inside discs tangent to $F$;
then one concludes exactly as for Theorems~\ref{thm-III-bis} and~\ref{thm.large-deviation-for-lamination}.
We are thus reduced to a statement analogous to Theorem~\ref{thm.volume-control-by-entropy}.

\begin{theoremalph}[Volume estimate]~\label{thm.volume-control-domination}
For any $C^1$ diffeomorphism $f$ and for any invariant compact set $\Lambda$ admitting a dominated splitting $E\oplus F$,
there exist a cone field $\cC^{F}$ which is a neighborhood of the bundle $F$, a neighborhood $U$ of $\Lambda$
and $r_0>0$ with the following property:
	for any $\mu\in\cM_{\mathrm{inv}}(\La,f)$ and $\e>0$, there exist $\eta,c>0$ such that
	for each compact disc $D\subset U$ tangent to $\cC^{F}$ with $\diam(D)<r_0$ and each $n\in\mathbb{N}$, one has 
	$$\Leb_D\big(  C_{n}(\mu,\eta)\cap D\cap \bigcap_{i=0}^{n-1}f^{-i}(U)\big)<c\cdot \exp\bigg(n\bigg(h_\mu(f)-\int\log{|\det(Df|_{F})|}\ud\mu+\e\bigg)\bigg).$$
	\end{theoremalph}

\begin{proof}[Sketch of the proof of Theorem~\ref{thm.volume-control-domination}]
As in the partially hyperbolic case,
we extend continuously the bundles $E,F$.
This allows to define a cone field $\cC^F$ on a neighborhood $U$ of $\Lambda$
which is a neighborhood of the bundle $F$.
We then consider
$F$-discs $D$, i.e. discs with the dimension of $F$ that are tangent to $\cC^F$,
and whose diameter is smaller than a small constant $r_0>0$.

Let $\mu$ be an invariant measure supported on $\Lambda$ and $\varepsilon>0$.
For any $\eta,\rho>0$ small, let us consider the $(n,\eta)$-convergent set
$C_n(\mu,\eta)$ of $\mu$ as in Section~\ref{s.volume-estimate} and let  $X_{n,\rho}$ be a $(n,\rho)$-separated subset
with maximal cardinal in
$$C_{n}(\mu,\eta)\cap D\cap\bigcap_{i=0}^{n-1}f^{-i}(U).$$
As in~\eqref{e1} (proof of Proposition~\ref{p.volume-through-separating-set}),
there exists $c_\varepsilon>0$ (only depending on $\varepsilon>0$) such that
$$
\Leb_D\big(C_{n}(\mu,\eta)\cap D\cap\bigcap_{i=0}^{n-1}f^{-i}(U)\big)\leq c_\e
\cdot e^{n(-\int\log{|\det(Df|_{F})|}\ud\mu+\e)}\cdot\#X_{n,\rho}.
$$

The proof of the variational principle (for a homeomorphism on a compact metric space)
gives the following estimate of $\#X_{n,\rho}$ (see~\cite[Lemma 5.2]{KH})
and concludes the proof.
\end{proof}

\begin{Lemma}
For any invariant measure $\mu$, and $\varepsilon,\rho>0$,
there exist $\eta>0$ and $n_0\geq 1$ with the following property.
If $X$ is a $(n,\rho)$-separated set with $n\geq n_0$
and
$$d\bigg(\;{\textstyle \frac{1}{\# X}} \sum_{x\in X} {\textstyle \frac 1 n}\sum_{k=0}^{n-1} \delta_{f^k(x)}
\;,\; \mu\;\bigg)<\eta,$$
then the cardinal $\# X$ is bounded by $\exp(n(h_\mu(f)+\varepsilon))$.
\end{Lemma}

\section{Large deviations for singular hyperbolic attractors}\label{Sec:flow}

Let $X$ be a $C^1$ vector field on $M$ and $(\phi_t)_{t\in\RR}$ be the flow generated by $X$. An attracting set $\Lambda$
is said to be \emph{singular hyperbolic}, if any singularity in $\Lambda$ is hyperbolic, and the time-one map $\phi_1$ admits a partially hyperbolic splitting  $T_\Lambda M=E^{ss}\oplus E^{cu}$ such that $E^{cu}$ is sectionally expanded
(there exists $t>0$ such that for any $x\in\Lambda$ the area along any $2$-plane $E\subset E^{cu}$ increases exponentially when one takes the image by $D\varphi_t$).
An SRB measure for  $(\phi_t)_{t\in\RR}$ is a probability measure which is preserved by the flow and which is SRB for
$\phi_1$ (it is then SRB for any $\phi_t$, $t>0$).

The previous statements allow to recover and improve a bit the results of~\cite{LeYa}.

\begin{Corollary}\label{Cor:basin-singular-hyperbolic}
For any $C^1$ vector field $X$, any singular hyperbolic attracting set $\Lambda$ supports an SRB measure.
More precisely, for Lebesgue almost every point $x$ in the basin of $\Lambda$, any limit measure $\mu\in\cM(x)$
is an SRB measure.
\end{Corollary}
\begin{proof}
Theorem~\ref{thm.volume-bound} applied to $\phi_1$ shows that for Lebesgue almost every point $x$ in the basin of $\Lambda$, any limit measure $\nu_0\in\cM(x)$ satisfies
$$h_{\nu_0}(\phi_1)\geq \int\log |\det D\phi_1|_{E^{cu}}|\ud{\nu_0}.$$ 
Ruelle's inequality and the singular hyperbolicity give
$h_{\nu}(\phi_1)\leq \int\log |\det D\phi_1|_{E^{cu}}|\ud{\nu}$ for any invariant measure $\nu$. Hence for
any ergodic component of $\nu$ of $\nu_0$ and $\nu$-a.e. point $z$, one has
$$h_\nu(\phi_1)= \int\log |\det D\phi_1|_{E^{cu}}|\ud\nu=\sum\lambda^+(z).$$
The $\phi$-invariant measure $\mu=\int_0^1(\phi_s)_*(\nu){\rm d}s$ satisfies the same formula and is SRB for $(\phi_t)_{t\in\RR}$.
\end{proof}

With  higher regularity,  we also obtain the uniqueness of the SRB measure.
\begin{theoremalph}\label{Thm:unique-srb-flow}
Let $X$ be a $C^{1+\alpha}$ vector field. Then any
singular hyperbolic transitive attractor $\Lambda$ supports a unique SRB measure
$\mu$. Its basin has full Lebesgue measure in the basin of $\La.$
\end{theoremalph}
\begin{proof}
Corollary~\ref{Cor:basin-singular-hyperbolic} gives the existence.

Let $\mu$ an SRB measure: the singular hyperbolicity implies that $\mu$ is a hyperbolic measure of the flow.
More precisely, to $\mu$-almost every point $x$ is associated its center-unstable set $W^{cu}(x)$,
which is the set of points $y$ such that there exists an increasing homeomorphism $h$ of $\mathbb{R}$
satisfying $d(\phi_t(y),\phi_{h(t)}(x))\to 0$ as $t\to -\infty$. This is an immersed submanifold tangent to $E^{cu}_x$
that is foliated by unstable leaves $W^u(y)$ which are one-codimensional in $W^{cu}(x)$.
The unstable leaves are the images of $W^u(x)$ by the flow. Hence the unstable foliation
is Lipschitz inside the center-unstable leaves of $\mu$.
Applying~\cite{LY} to the diffeomorphism $\phi_1$, the disintegration of $\mu$ along the unstable leaves is equivalent to the Lebesgue measure:
the statement is given for $C^2$ diffeomorphisms, but the proof only uses a $C^{1+\alpha}$-regularity,
once one knows that the unstable lamination is Lipschitz along the center-unstable direction
see~\cite[Theorem A and Section 4.2]{LY}.

Note that any ergodic component of $\mu$ is still an SRB measure,
one will thus assume that $\mu$ is ergodic.
For $\mu$-almost every point $x$, the forward orbit of Lebesgue almost every point $y\in W^u(x)$
equidistributes towards $\mu$. Since $W^{cu}(x)$ can be obtained by flowing the unstable manifold $W^u(x)$,
one deduces that the forward orbit of Lebesgue almost every point $y\in W^{cu}(x)$
equidistributes towards $\mu$.
The same proof as for Proposition~\ref{Prop:srb-positive-lyapunov}
shows that there exists a non-empty open set $U$ which intersects the support of $\mu$
and has the property that the forward orbit of Lebesgue almost every point in $U$ equidistributes towards $\mu$.

If $\mu_1,\mu_2$ are two (ergodic) SRB measures supported on $\Lambda$,
one associates two open sets $U_1,U_2$.
The transitivity of $\Lambda$ implies that there exists a non-empty open subset $V\subset U_1$
having a forward iterate in $U_2$. Hence the forward orbit of almost every point in $V$ equidistributes towards $\mu_1$ and $\mu_2$.
This gives $\mu_1=\mu_2$, hence the uniqueness of the SRB measure.

By Corollary ~\ref{Cor:basin-singular-hyperbolic}, for   Lebesgue a.e. point $x$ in the basin of $\Lambda$, each limit measure of 
$$\frac{1}{t}\int_0^t\delta_{\phi_s(x)}\ud s$$
(as $t\to +\infty$) is an SRB measure. The uniqueness of the SRB measure implies that its basin has full Lebesgue measure in the basin of $\La$.
\end{proof}

We also deduce a large deviation estimate.
\begin{theoremalph}~\label{thm.flow-deviation}
Let $X$ be a $C^1$ vector field and $\La$ be a singular hyperbolic attracting set admitting a unique ergodic SRB measure $\mu$.
Then there exists a neighborhood $U$ of $\La$ such that for any continuous function $\psi:M\to\mathbb{R}$ and any $\e>0$, there exist $a_\e>0$ and $b_\e>0$ such that 
	$$\Leb\bigg(\bigg\{x\in U: \bigg|\frac{1}{t}\int_{0}^t\psi(\phi_s(x))\ud s-\int \psi\ud\mu\bigg|\geq\e\bigg\}\bigg)<a_\e\cdot e^{-tb_\e}\textrm{\: for any $t\in\mathbb{R}^+$}.$$
\end{theoremalph}
\begin{proof}
Let $U$ be an attracting neighborhood of $\Lambda$ in its basin.
From the continuity of $\psi$,
for any $\varepsilon>0$ there exists $t_\e$ large such that for any $t_0>0$ small enough,
any $t>t_\e$ and any $x\in M$, denoting $n=[t/t_0]$ one has
\begin{equation}\label{e.continuity}
\bigg|\frac 1 t \int_0^t\psi(\phi_s(x))\ud s -  \frac 1 n \sum_{i=0}^{n-1} \psi\circ \phi_{it_0}(x) \bigg|<\e/2.
\end{equation}
One can assume that $t_0$ avoids a countable set and by~\cite{PuSh}
the measure $\mu$ is ergodic for the map $\phi_{t_0}$.
We also notice that it is the unique SRB for $\phi_{t_0}$ on $\Lambda$.
Indeed if $\nu$ is an SRB measure of $\phi_{t_0}$, then $\frac{1}{t_0}\int_{0}^{t_0}(\phi_s)_*\nu\ud s$ is an SRB measure for both $\phi_{t_0}$ and $(\phi_t)_{t\in\RR}$. 
Hence $\frac{1}{t_0}\int_{0}^{t_0}(\phi_s)_*\nu\ud s=\mu$.
But since $\mu$ is ergodic it is an extremal point of the set of $\phi_{t_0}$-invariant probability measures.
As a consequence $(\phi_t)_*\nu=\mu$ for any $t\in \RR$ and in particular $\nu=\mu$.

We can then apply Theorem~\ref{thm.deviation-domination} to $\Lambda$ and $\phi_{t_0}$:
there exists a neighborhood $U_\e$ of $\Lambda$ and $a,b_\e>0$ such that for any $n\in\mathbb{N}$

$$\Leb\bigg\{ x\in \bigcap_{i=0}^{n-1}\phi_{-it_0}(U): \ud\bigg(\frac{1}{n}\sum_{i=0}^{n-1}\psi(\phi_{it_0}(x)), \int \psi \ud\mu\bigg)\geq\e/2\bigg\}
<a\cdot e^{-n b_\e}.$$
Combined with~\eqref{e.continuity}, this gives the result for points in $\bigcap_{t\geq  0}\phi_{-t}(U)$ and $n\geq t/t_0$.
Since $U$ is an attracting neighborhood of $\Lambda$ in its basin, there exists $N\geq 1$ such that
for any $x\in U$, the image $\phi_N(x)$ belongs to $\bigcap_{t\geq  0}\phi_{-t}(U)$.
One thus concludes the large deviation estimate for any $t>0$ by considering $a_\e$ large enough.
\end{proof}

\vskip 5pt

\begin{tabular}{l l l}
\emph{\normalsize Sylvain Crovisier}
& \quad &
\emph{\normalsize Dawei Yang}
\medskip\\

\small Laboratoire de Math\'ematiques d'Orsay
&& \small School of Mathematical Sciences\\
\small CNRS - Universit\'e Paris-Sud
&& \small Soochow University\\
\small Orsay 91405, France
&& \small Suzhou, 215006, P.R. China\\
\small \texttt{Sylvain.Crovisier@math.u-psud.fr}
&& \small \texttt{yangdw1981@gmail.com}\\
&& \small \texttt{yangdw@suda.edu.cn}
\medskip\\
\emph{\normalsize Jinhua Zhang}
\medskip\\

\small Laboratoire de Math\'ematiques d'Orsay
 \\
\small CNRS - Universit\'e Paris-Sud
 \\
\small Orsay 91405, France
\\
\small \texttt{jinhua.zhang@math.u-psud.fr, zjh200889@gmail.com}
 
\end{tabular}

\end{document}